\documentclass[a4paper,12pt]{article}

\usepackage{latexsym}
\usepackage{mathrsfs}
\usepackage{amsfonts,amsmath,amssymb}
\usepackage{indentfirst}
\usepackage{subeqnarray}
\usepackage{verbatim}
\usepackage{epstopdf}
\usepackage{algorithm}
\usepackage{algorithmic}
\usepackage{graphicx}
\usepackage{subfigure}

\newtheorem{theorem}{Theorem}[section]
\newtheorem{example}[theorem]{Example}
\newtheorem{lemma}[theorem]{Lemma}
\newtheorem{assumption}[theorem]{Assumption}

\newtheorem{definition}[theorem]{Definition}

\newtheorem{proposition}[theorem]{Proposition}
\newtheorem{remark}[theorem]{Remark}

\numberwithin{equation}{section}
\newenvironment{proof}[1][Proof]{\textbf{#1.} }
{\ \rule{0.75em}{0.75em}\smallskip}

\begin{document}

\begin{center}
\Large\bf $\alpha\ell_{1}-\beta\ell_{2}$ sparsity regularization for nonlinear ill-posed problems
\end{center}

\begin{center}
Liang Ding\footnote{Department of Mathematics, Northeast Forestry University, Harbin 150040, China;
e-mail: {\tt dl@nefu.edu.cn}. The work of this author was supported by the Fundamental Research Funds
for the Central Universities (no.\ 2572018BC02), Heilongjiang Postdoctoral Research Developmental Fund
(no.\ LBH-Q16008), the National Nature Science Foundation of China (no.\ 41304093).}\quad and \quad
Weimin Han\footnote{Department of Mathematics, University of Iowa, Iowa City, IA 52242, USA;
e-mail: {\tt weimin-han@uiowa.edu}.}
\end{center}

\medskip
\begin{quote}
{\bf Abstract.} In this paper, we consider the $\alpha\| \cdot\|_{\ell_1}-\beta\| \cdot\|_{\ell_2}$
sparsity regularization with parameter $\alpha\geq\beta\geq0$ for nonlinear ill-posed inverse problems.
We investigate the well-posedness of the regularization. Compared to the case where $\alpha>\beta\geq0$,
the results for the case $\alpha=\beta\geq0$ are weaker due to the lack of coercivity and Radon-Riesz property
of the regularization term. Under certain condition on the nonlinearity of $F$, we prove that
every minimizer of $ \alpha\| \cdot\|_{\ell_1}-\beta\| \cdot\|_{\ell_2}$
regularization is sparse. For the case $\alpha>\beta\geq0$, if the exact solution is sparse, we derive
convergence rate $O(\delta^{\frac{1}{2}})$ and $O(\delta)$ of the regularized solution under two commonly adopted
conditions on the nonlinearity of $F$, respectively.  In particular, it is shown that the iterative soft
thresholding algorithm can be utilized to solve
the $ \alpha\| \cdot\|_{\ell_1}-\beta\| \cdot\|_{\ell_2}$ regularization problem for nonlinear ill-posed equations.
Numerical results illustrate the efficiency of the proposed method.
\end{quote}


\section{Introduction}
The investigation of the non-convex
$\alpha\|\cdot\|_{\ell_1}-\beta\| \cdot\|_{\ell_2}$ $(\alpha\ge\beta\geq0)$ regularization has attracted attention in the field of sparse recovery over
the last five years, see \cite{DH19,LCGK20,LY18,YSX17,YLHX15} and references therein.
As an alternative of the $\ell_p$-norm with $0\leq p<1$,
the advantages of using the functional $ \alpha\|\cdot\|_{\ell_1}-\beta\| \cdot\|_{\ell_2}$ $(\alpha\geq\beta\geq0)$
lie in the fact that it is a good approximation of the
$\ell_0$-norm and it has a simpler structure than the $\ell_0$-norm from the perspective of computation.
Moreover, it is difficult to determine the optimal exponent $p$ for $\ell_p$ ($0\leq p<1)$ regularization (\cite{LR17}).
Nevertheless, for the functional $\alpha\|\cdot\|_{\ell_1}-\beta\|\cdot\|_{\ell_2}$, it can be shown that
$\eta=\beta/\alpha$ plays a role similar to that of $p$ in $\ell_p$ regularization.
In this paper, we investigate
the potential of the regularization method for solving nonlinear ill-posed operator equations with sparse solutions.
In addition, we analyze the well-posedness of the regularization for the particular case $\alpha=\beta$.

We are interested in solving an ill-posed operator equation of the form
\begin{equation}\label{equ1_1}
F(x)=y,
\end{equation}
where $x$ is sparse, $F: \ell_2 \rightarrow Y$ is a weakly sequentially closed nonlinear operator mapping between
the $\ell_2$ space and a Hilbert space $Y$ with norms $\|\cdot\|_{\ell_2}$ and $\|\cdot\|_{Y}$, respectively.
Throughout this paper, we let $\langle \cdot,\cdot\rangle$ denote the inner product in the $\ell_2$ space and
$e_i=(\underbrace{0,\cdots,0,1}_i,0,\cdots)$. The exact data $y^{\dag}$ and the observed data $y^{\delta}$ satisfy
$\|y^{\delta}-y^{\dag}\|_Y\leq \delta$ with a noise level $\delta>0$. The most commonly adopted technique to
solve the problem (\ref{equ1_1}) is sparsity regularization, see the monographs \cite{F2010,SGGHL2009} and the special
issues \cite{BB18,DDD16,JM12,JMS17} for many developments on regularizing properties and minimization schemes.

The first theoretical analysis on sparsity regularization for ill-posed inverse problems dates back to 2004.
In the seminal paper \cite{DDD04}, Daubechies et al proposed an $\ell_p$ ($1\leq p\leq 2$) sparsity regularization
for linear ill-posed problems and established the convergence of an iterative soft thresholding algorithm.
Inspired by \cite{DDD04}, many investigations focused on the regularizing properties and
iteration schemes for linear ill-posed inverse problems, see \cite{BB18,F2010,SGGHL2009}. Subsequently,
the schemes and results were quickly extended to nonlinear ill-posed inverse problems. Much effort has been
devoted to investigating the regularization properties as well as the minimization of the sparsity regularization
for nonlinear ill-posed inverse problems, see \cite{JM12,J16,LMM12,RT05,RT06,TB10,ZW16} and the references therein.
We emphasize that in the above cited references only the convex case $p\geq 1$ is investigated. For the non-convex
case $0\leq p< 1$, particular conditions and techniques are needed to analyze the well-posedness and convergence rate.
In \cite{G09}, a sub-linear $\ell_q$ regularization is proposed and convergence is proved in the sense of the weak$^{*}$
topology on $\ell_1$. A multi-parameter Tikhonov regularization with $\ell_0$ constraint is presented in \cite{WLHC19,WLMC13}, where regularizing properties as well as convergence rate results are obtained. In \cite{Z09},
with the use of a superposition operator $\mathcal{N}_{p,q}$, the sparsity regularization with $0\leq p<1$ can be
studied within a more classical convex formulation with $1\leq q\leq 2$.  Then the well-known results on regularizing
properties of convex sparsity regularization can be utilized to analyze the original non-convex sparsity regularization.

Concerning the minimization of $\ell_p$ sparsity regularization with $0\leq p<1$, several numerical algorithms
were developed for linear ill-posed inverse problems, see \cite{BD09,GOY2016,IK14,WYZ19,XZWCL10,YLHX15}.
Unfortunately, the algorithms in these references, e.g.\ alternating direction method of multipliers (ADMM) (\cite{WYZ19}),
iteratively reweighted least squares (IRLS) (\cite{XZWCL10}), primal-dual active set method (\cite{IK14}) and
iterative hard thresholding (\cite{BD09}) can not be extended to nonlinear ill-posed equations directly.
Sparsity regularization with non-convex regularized term for nonlinear ill-posed inverse problems is far
from being investigated systematically, especially in computation. Though there is great potential in the
non-convex sparsity regularization for nonlinear ill-posed inverse problems, to the best of our knowledge,
only one result is available in the literature.  In \cite{RZ12},
the non-convex Tikhonov functional is transformed to a more viable one. Then a surrogate functional approach
is applied to the new convex functional straightforwardly.

In this paper, we solve the nonlinear
ill-posed inverse problem (\ref{equ1_1}) by the following regularization method:
\begin{equation}\label{equ1_2}
\min\mathcal{J}_{\alpha,\beta}^{\delta}(x)=\frac{1}{q}\| F(x)-y^{\delta}\|_Y^{q}+\mathcal{R}_{\alpha,\beta}(x),
\end{equation}
where $q\geq 1$ and
\begin{equation}\label{equ1_3}
\mathcal{R}_{\alpha,\beta}(x):=\alpha\|x\|_{\ell_1}-\beta\| x\|_{\ell_2},\quad \alpha\geq\beta\geq 0.
\end{equation}
For $\alpha>0$, denoting $\eta=\beta/\alpha$, we can equivalently express the functional in (\ref{equ1_3}) as
\[ \mathcal{R}_{\alpha,\beta}(x)=\alpha\,\mathcal{R}_{\eta}(x), \]
where $ \mathcal{R}_{\eta}(x):=\|x\|_{\ell_1}-\eta\| x\|_{\ell_2}$, $1\geq\eta\geq0$.
We will investigate the well-posedness of the problem (\ref{equ1_2}). For the case $\alpha>\beta\geq0$,
we show the existence, stability as well as convergence of regularized solutions under the assumption that
the nonlinear operator $F$ is weakly sequentially closed. The numerical experiments in \cite{DH19} show that
we can obtain satisfactory results even when $\alpha=\beta$. Actually, $\mathcal{R}_{\alpha,\beta}(x)$
behaves more and more like the $\ell_0$-norm as $\beta/\alpha\rightarrow 1$. So in this paper, we also analyze properties of ${\cal R}_{\alpha,\beta}(x)$ when $\alpha=\beta$, even though the well-posedness results of the regularization are weaker than that in the case $\alpha>\beta\ge 0$. For the case
$\alpha>\beta\geq0$, we identify the convergence rate under an appropriate source condition. As is standard
in analyzing convergence rates, we need to impose restrictions on the nonlinearity
of the operator $F$. Typically, the restrictions are utilized to bound the crucial term
$\langle F'(x^{\dag})(x-x^{\dag}),\omega_i\rangle$ in deriving convergence rate results. Under two commonly adopted conditions on the nonlinearity of $F$, we get convergence rates $O(\delta^{\frac{1}{2}})$ and $O(\delta)$ of the regularized solution in the $\ell_2$-norm, respectively.

For the minimization problem (\ref{equ1_2}), we propose an iterative soft thresholding algorithm (\cite{AM09,DDD04})
based on the generalized conditional gradient method (GCGM). In \cite{BBLM07,BLM09}, GCGM is applied to solve the minimization problem for sparsity regularization with the convex regularization term $\sum\limits_n w_n|\langle u, \phi_n\rangle|^p$ with $p\geq 1$, where $\{w_{n}>0\}$ are the weights, and $\{\phi_{n}\}$ is an
orthonormal basis of a Hilbert space. In this paper, it is shown that this method can be applied to the non-convex
$\alpha\ell_{1}-\beta\ell_{2}$ sparsity regularization for nonlinear inverse problems. For the case $q=2$, we rewrite the functional
$\mathcal{J}_{\alpha,\beta}^{\delta}$ in (\ref{equ1_2}) as
\[ \mathcal{J}_{\alpha,\beta}^{\delta}(x)=G(x)+\Phi(x), \]
where $G(x)=(1/2)\,\|F(x)-y^{\delta}\|_Y^2-\Theta(x)$, $\Phi(x)=\Theta(x)+\alpha\|x\|_{\ell_1}-\beta\|x\|_{\ell_2}$
and $\Theta(x)=(\lambda/2)\,\|x\|_{\ell_2}^2+\beta\|x\|_{\ell_2}$.
We show that if the nonlinear operator $F$ is continuously Fr\'echet differentiable and $F$ is bounded on bounded sets,
then the iterative soft thresholding algorithm is convergent.

The rest of the paper is organized as follows. In Section \ref{sec2}, we analyze the well-posedness
of the $\alpha\|\cdot\|_{\ell_1}-\beta\|\cdot\|_{\ell_2}$ $(\alpha\geq\beta\geq0)$ regularization. In Section \ref{sec3},
we derive the convergence rates in the $\ell_2$-norm under an appropriate source condition and two commonly adopted
conditions on the nonlinearity of $F$. In Section \ref{sec4}, we present an iterative soft thresholding algorithm
based on GCGM and discuss its convergence. Finally, some numerical experiments are presented in Section \ref{sec5}.


\section{Well-posedness of regularization problem}\label{sec2}

In this section we analyze the well-posedness of the regularization method, i.e., existence, stability
as well as convergence of regularized solutions. For the case $\alpha=\beta$, $\mathcal{R}_{\alpha,\beta}(x)$
does not have coercivity nor Radon-Riesz property, and the well-posedness result of the regularization is
weaker than that in the case $\alpha>\beta$.


Let us denote a minimizer of the functional $\mathcal{J}_{\alpha,\beta}^{\delta}(x)$ by $x^{\delta}_{\alpha,\beta}$, i.e.
\begin{equation}\label{equ2_1}
 \begin{array}{llc}
\displaystyle x^{\delta}_{\alpha,\beta}= \arg\min\limits_x \mathcal{J}_{\alpha,\beta}^{\delta}(x),\quad                 \mathcal{J}_{\alpha,\beta}^{\delta}(x)=\frac{1}{q}\|F(x)-y^{\delta}\|_Y^{q}+\mathcal{R}_{\alpha,\beta}(x)
 \end{array}
 \end{equation}
The $\mathcal{R}_{\eta}$-minimum solution is defined next.

\begin{definition}\label{Def1}
An element $x^{\dagger}\in \ell_2$ is called an $\mathcal R_{\eta}$-minimum solution to the
problem \eqref{equ1_1} if it satisfies
\[\displaystyle F(x^{\dagger})=y~~and~~\displaystyle \mathcal R_{\eta}(x^{\dagger})
=\min\left\{\mathcal R_{\eta}(x)\mid x\in \ell_2,\,F(x)=y\right\}.\]
\end{definition}

\begin{definition}\label{Def2}
$x\in \ell_2$ is called sparse if $\mathrm{supp}(x):=\{i\in\mathbb{N}\mid x_{i}\neq0\}$ is finite,
where $x_i$ is the $i^{\rm th}$ component of $x$.
\end{definition}

To characterize the sparsity, as in \cite{DDD04}, we define the index set
\begin{equation}\label{equ2_1a}
I(x^{\dag})=\{i\in \mathbb{N}\mid x_{i}^{\dag}\neq 0\},
\end{equation}
where $x_{i}^{\dag}$ is the $i^{\rm th}$ component of $x^{\dag}$.

\subsection{The case $\alpha>\beta\geq0$}

First, in Lemma \ref{lemma2_0} we recall some properties of $\mathcal{R}_{\alpha,\beta}(x)$ which are crucial tools in analyzing the well-posedness of regularization, see \cite{DH19} for the proofs.

\begin{lemma}\label{lemma2_0} The functional $\mathcal{R}_{\alpha,\beta}(x)$ has the following properties:

{\rm(i)  (Coercivity)} For $x\in \ell_2$, $\|x\|_{\ell_2}\rightarrow \infty$ implies $\mathcal{R}_{\alpha,\beta}(x)\rightarrow \infty$.

{\rm(ii)  (Weak lower semi-continuity)} If $x_n\rightharpoonup x$ in $\ell_2$ and $\{\mathcal{R}_{\alpha,\beta}(x_n)\}$
is bounded, then
\[\liminf_n\mathcal{R}_{\alpha,\beta}(x_n)\geq\mathcal{R}_{\alpha,\beta}(x).\]

{\rm (iii) (Radon-Riesz property)} If $x_n\rightharpoonup x$ in $\ell_2$ and $\mathcal{R}_{\alpha,\beta}(x_n)\rightarrow
\mathcal{R}_{\alpha,\beta}(x)$, then $\|x_n-x\|_{\ell_2}\rightarrow 0$.
\end{lemma}

\begin{lemma}\label{lemma2_1}
Assume the sequence $\{\|y_n\|_{Y}\}$ is bounded in $Y$. For a given $M>0$, let $\{x_n\}\in \ell_2$ and
\begin{equation}\label{equlemma2_1}
\frac{1}{q}\| F(x_{n})-y_{n}\|_Y^{q} +\mathcal{R}_{\alpha,\beta}(x_n)\leq M.
\end{equation}
Then there exist an $x\in \ell_2$ and a subsequence $\{x_{n_k}\}$ of $\{x_n\}$ such that $x_{n_k}\rightharpoonup x$ and
$F(x_{n_k})\rightharpoonup F(x)$.
\end{lemma}
\begin{proof}
By \eqref{equlemma2_1}, $\{\mathcal{R}_{\alpha,\beta}(x_n)\}$ is bounded. It follows from the coercivity of $\mathcal{R}_{\alpha,\beta}(x)$ that $\{\|x_n\|_{\ell_2}\}$ is bounded.
Meanwhile, since $\{\|y_n\|_{Y}\}$ is bounded, $\{\|F(x_{n})\|_{Y}\}$ is bounded. Hence, there exists a subsequence
$\{x_{n_k}\}$ of $\{x_{n}\}$, $x\in \ell_2$ and $y\in Y$ such that
\[ x_{n_k}\rightharpoonup x\ {\rm in\ }\ell_2, \quad F(x_{n_k})\rightharpoonup y\ {\rm in\ }Y .\]
Since $F$ is weakly sequentially closed, $F(x)=y$. This proves the lemma.
\end{proof}


We have the existence, stability as well as convergence of the regularized solution given in the next three results, similar to Theorems 2.11, 2.12 and 2.13 in \cite{DH19}. Their proofs are based on the properties stated in Lemmas \ref{lemma2_0} and \ref{lemma2_1}.

\begin{theorem} {\rm (Existence)}   
For any $y^{\delta}\in Y$, there exists at least one minimizer to $\mathcal{J}_{\alpha,\beta}^{\delta}(x)$ in $\ell_2$.
\end{theorem}

\begin{theorem}\label{theory2} {\rm (Stability)}
Let $\alpha_n> \beta_n\geq0$, $\alpha_n\rightarrow \alpha$, $\beta_n\rightarrow \beta$ as $n\rightarrow \infty$.
Let the sequence $\{y_n\}\subset Y$ be convergent to $y^{\delta}\in Y$, and let $x_n$ be a minimizer to
$\mathcal{J}_{\alpha_n,\beta_n}^{\delta_n}(x)$. Then the sequence $\{x_n\}$ contains a subsequence converging
to a minimizer of $\mathcal{J}_{\alpha,\beta}^{\delta}(x)$. Furthermore, if $\mathcal{J}_{\alpha,\beta}^{\delta}(\cdot)$
has a unique minimizer $x_{\alpha,\beta}^{\delta}$, then
$\lim_{k \rightarrow \infty} \| x_{n_k}-x_{\alpha,\beta}^{\delta} \|_{\ell_2}=0$.
\end{theorem}

\begin{theorem}\label{theory3} {\rm (Convergence)}
Let $\alpha_n:=\alpha(\delta_n)$, $\beta_{n}:=\beta(\delta_{n})$, $\alpha_n>\beta_n\ge 0$ satisfy
\[ \lim_{n\to\infty}\alpha_n=0, \quad\lim_{n\to\infty}\beta_n=0\quad and\quad
\lim_{n\to\infty}\frac{\delta_n^{q}}{\alpha_n}=0.\]
Assume that $\displaystyle\eta=\lim_{n\to\infty}\eta_n\in [0,1)$ exists, where
$\eta_n=\beta_n/\alpha_n$.
Let $\delta_n\rightarrow 0$ as $n\rightarrow +\infty$ and $y^{\delta_n}$ satisfy $\| y-y^{\delta_n}\|\leq\delta_n$.
Moreover, let
\[  \displaystyle x_{\alpha_{n},\beta_{n}}^{\delta_{n}} \in
\arg\min\mathcal{J}_{\alpha_n,\beta_n}^{\delta_n}(x).\]
Then $\{x_{\alpha_{n},\beta_{n}}^{\delta_{n}}\}$ has a subsequence, still denoted by $\{x_{\alpha_n,\beta_n}^{\delta_n}\}$,
converging to an $\mathcal{R}_{\eta}$-minimizing solution $x^{\dag}$ in $\ell_2$.
Furthermore, if the $\mathcal{R}_{\eta}$-minimizing solution $x^{\dag}$ is unique, then
\[\lim_{n \rightarrow +\infty} \| x_{\alpha_{n},\beta_n}^{\delta_{n}}-x^{\dag}\|_{\ell_2}=0.\]
\end{theorem}


\subsection{The case $\alpha=\beta>0$}

We turn to the case $\alpha=\beta>0$. The functional $\mathcal{R}_{\alpha,\beta}(x)$ remains to be weakly lower semi-continuous, see \cite[Lemma 2.8, Remark 2.9]{DH19} for details. However, coercivity and Radon-Riesz property cannot be extended to the case $\alpha=\beta$, see Examples \ref{ex2.9} and \ref{ex2.12} below.

\begin{example}{\rm (Non-coercivity)}\label{ex2.9}
Let $x=(x_1, x_2, \cdots, x_n, 0, 0, 0, \cdots)$. If $x_1\rightarrow \infty$ and $\sum\limits_{i=2}^{n}|x_i|$ is bounded, then
$\|x\|_{\ell_2}\rightarrow\infty$. We have
\begin{align*}
\displaystyle \mathcal{R}_{\alpha,\alpha}(x)&= \alpha\|x\|_{\ell_1}-\alpha\| x\|_{\ell_2}=\alpha\frac{\|x\|_{\ell_1}^2-\| x\|_{\ell_2}^2}{\|x\|_{\ell_1}+\| x\|_{\ell_2}}\leq\alpha\frac{|x_1|\left(\sum\limits_{i=2}^{n}|x_i|\right)+|x_2|\left(\sum\limits_{i=3}^{n}|x_i|\right)+\cdots+|x_{n-1}||x_{n}|}{2|x_1|+|x_2|+\cdots+|x_n|}.
\end{align*}
Thus, $\limsup_{x_1\rightarrow \infty}\mathcal{R}_{\alpha,\alpha}(x)\leq\frac{\alpha}{2}\sum\limits_{i=2}^{n}|x_i|$. So $\mathcal{R}_{\alpha,\alpha}(x)$ is not coercive.
\end{example}

Note that the standard proof of the well-posedness of Tikhonov regularization is invalid without the coercivity of $\mathcal{R}_{\alpha,\alpha}(x)$. So to ensure the well-posedness of the problem \eqref{equ1_2} in the case $\alpha=\beta$, we provide a result next where an additional restriction, i.e. coercivity  is imposed on the nonlinear operator $F$, see \cite{AZ00,IK14} for some examples of the nonlinear (or linear) coercive operator.

\begin{lemma}\label{lemmaadd1} Assume $F(x)$ is coercive with respect to $\|x\|_{\ell_2}$, i.e. $\|x\|_{\ell_2}\rightarrow \infty$ implies $\|F(x)\|_Y\rightarrow \infty$. Then the functional
$\mathcal{J}_{\alpha,\alpha}^{\delta}(x)$ is coercive.
\end{lemma}
\begin{proof}
By the definition of $\mathcal{J}_{\alpha,\beta}^{\delta}(x)$,
\[  \mathcal{J}_{\alpha,\alpha}^{\delta}(x)= \frac{1}{q}\| F(x)-y^{\delta}\|_Y^{q}+ \alpha\|x\|_{\ell_1}-\alpha\| x\|_{\ell_2}\geq\frac{1}{q}\big|\| F(x)\|_Y-\|y^{\delta}\|_Y\big|^{q}.\]
Since $F(x)$ is coercive, it is obvious that $\mathcal{J}_{\alpha,\alpha}^{\delta}(x)\rightarrow \infty$ as $\|x\|_{\ell_2}\rightarrow \infty$.
\end{proof}

Based on Lemma \ref{lemmaadd1}, we can demonstrate the existence of the regularized solution; the proof is similar to that in Theorem 2.11 in \cite{DH19}. Next we give an example to show that $x_n$ does not necessarily converge strongly to $x$ even if $x_n\rightharpoonup x$ in $\ell_2$ and $\mathcal{R}_{\alpha,\alpha}(x_n)\rightarrow \mathcal{R}_{\alpha,\alpha}(x)$. Thus $\mathcal{R}_{\alpha,\alpha}(x)$ fails to satisfy the Radon-Riesz property.

\begin{example}{\rm (Non-Radon-Riesz property)}\label{ex2.12}
Let $x_n=(\underbrace{0,\cdots,0,1}_n,0,\cdots)$ and $x=0$, then $x_n\rightharpoonup x$ in $\ell_2$. We have
\[ \mathcal{R}_{\alpha,\alpha}(x_n)=\alpha(\|x_n\|_{\ell_1}-\|x_n\|_{\ell_2})=0 \quad and \quad
 \mathcal{R}_{\alpha,\alpha}(x)=0. \]
So $\mathcal{R}_{\alpha,\alpha}(x_n)\rightarrow \mathcal{R}_{\alpha,\alpha}(x)$. However, $\|x_n-x\|_{\ell_2}=1$, which implies that $x_n$ does not converge strongly to $x$.
\end{example}

Since $\mathcal{R}_{\alpha,\alpha}(x)$ fails to satisfy the Radon-Riesz property, the standard proof of the well-posedness can not ensure the strong convergence. Without the Radon-Riesz property, we may expect to have only weak convergence
in stability and convergence properties of the regularized solution. With Lemma \ref{lemmaadd1}, the proof of stability
and convergence is similar to that of Theorems 2.12 and 2.13 in \cite{DH19}.


\subsection{Sparsity}

Next we turn to a discussion of the sparsity of the regularization solution. Under a restriction on the nonlinearity of $F$, it can be shown that every minimizer of $\mathcal{J}_{\alpha,\beta}^{\delta}(x)$ is sparse whenever $\alpha>\beta$ or $\alpha=\beta$.

\begin{proposition}\label{lemmasparsity} {\rm (Sparsity)}
Let $x$ be a minimizer of $\mathcal{J}_{\alpha,\beta}^{\delta}(x)$ $(\alpha\geq\beta\geq0)$. Assume that $F$ has a continuous Fr\'{e}chet derivative and there exists a constant $\gamma>0$ such that
\begin{equation}\label{equ2_9}
\displaystyle \|F'(y)-F'(x)\|_{L(\ell_2, Y)}\leq \gamma\|y-x\|_{\ell_2}
 \end{equation}
for any $y\in B_{\delta}(x)$, where $B_{\delta}(x):=\{y \mid \|y-x\|_{\ell_2}\leq \delta\}$, $\delta\geq\|x\|_{\infty}$. Then $x$ is sparse.
\end{proposition}
\begin{proof}
For simplicity,
we only discuss the case $q=2$. For $i\in \mathbb{N}$, consider $\bar{x}:=x-x_ie_i$, where $x_{i}$ is the $i^{\rm th}$ component of $x$. It is clear that $\bar{x}\in B_{\delta}(x)$. By the definition of $x$,
\begin{equation}\label{sparsity0}
\displaystyle \frac{1}{2}\| F(x)-y^{\delta}\|_Y^{2}+\mathcal{R}_{\alpha,\beta}(x)
\leq \frac{1}{2}\| F(\bar{x})-y^{\delta}\|_Y^{2}+\mathcal{R}_{\alpha,\beta}(\bar{x}).
\end{equation}
If $x=0$, then $x$ is sparse. Suppose $x\neq0$. By (\ref{sparsity0}), we see that
\begin{align}
\alpha|x_i|-\beta \frac{|x_i|^2}{\|x\|_{\ell_2}+\|\bar{x}\|_{\ell_2}}
&=\mathcal{R}_{\alpha,\beta}(x)-\mathcal{R}_{\alpha,\beta}(\bar{x})\nonumber\\
&\leq \frac{1}{2}\| F(\bar{x})-y^{\delta}\|_Y^{2}-\frac{1}{2}\| F(x)-y^{\delta}\|_Y^{2}\nonumber\\
&= \frac{1}{2}\| F(\bar{x})-F(x)\|_Y^{2}+\langle F(x)-y^{\delta}, F(\bar{x})-F(x)\rangle. \label{sparsity1}
\end{align}
Note that (\ref{equ2_9}) implies that
\begin{equation}\label{sparsity2}
F(\bar{x})=F(x)+F'(x)(\bar{x}-x)+r_{\alpha}^{\delta}
\end{equation}
with
\begin{equation}\label{sparsity3}
\|r_{\alpha}^{\delta}\|_Y\leq \frac{\gamma}{2}\|\bar{x}-x\|_{\ell_2}^2,
\end{equation}
see \cite[p.\ 14]{JM12} for a proof of this result. A combination of (\ref{sparsity2}) and (\ref{sparsity3}) implies that
\begin{align}
\|F(\bar{x})-F(x)\|_Y^2 & = \|F'(x)(\bar{x}-x)\|_Y^2+\|r_{\alpha}^{\delta}\|_Y^2+2\langle F'(x)(\bar{x}-x),r_{\alpha}^{\delta} \rangle\nonumber\\
& \leq \|F'(x)\|_{L(\ell_2,Y)}^2\|\bar{x}-x\|_{\ell_2}^2+\frac{\gamma^2}{4}\|\bar{x}-x\|_{\ell_2}^4+\gamma\|F'(x)\|_{L(\ell_2,Y)}\|\bar{x}-x\|_{\ell_2}^3\nonumber\\
& =|x_i|^2\|F'(x)\|_{L(\ell_2,Y)}^2+\frac{\gamma^2}{4}|x_i|^4+\gamma|x_i|^3\|F'(x)\|_{L(\ell_2,Y)}.\label{sparsity4}
\end{align}
Moreover,
\begin{align}
\langle F(x)-y^{\delta}, F(\bar{x})-F(x)\rangle &= \langle F(x)-y^{\delta}, F'(x)(\bar{x}-x)+r_{\alpha}^{\delta}\rangle\nonumber\\
& \leq -x_i\langle F'(x)^{*}(F(x)-y^{\delta}), e_i\rangle +\frac{\gamma}{2}|x_i|^2\|F(x)-y^{\delta}\|_Y\label{sparsity5}.
\end{align}
A combination of (\ref{sparsity1}), (\ref{sparsity4}) and (\ref{sparsity5}) implies that
\begin{align}
\alpha|x_i|-\beta \frac{|x_i|^2}{\|x\|_{\ell_2}+\|\bar{x}\|_{\ell_2}}\leq &\frac{1}{2}|x_i|^2\|F'(x)\|_{L(\ell_2,Y)}^2+\frac{\gamma^2}{8}|x_i|^4+\frac{1}{2}\gamma|x_i|^3\|F'(x)\|_{L(\ell_2,Y)}\nonumber\\
&-x_i\langle F'(x)^{*}(F(x)-y^{\delta}), e_i\rangle +\frac{\gamma}{2}|x_i|^2\|F(x)-y^{\delta}\|_Y\label{sparsity6}
\end{align}
for every $i\in \mathbb{N}$.
Now if $\|x\|_0=1$, then $x$ is sparse. Otherwise, $\|x\|_0\geq2$ and then $\frac{|x_i|}{\|x\|_{\ell_2}
+\|\bar{x}\|_{\ell_2}}< 1$. Thus, there exists a constant $c>0$ such that
\begin{equation}\label{addeq1}
\frac{c+\eta|x_i|}{\|x\|_{\ell_2}
+\|\bar{x}\|_{\ell_2}}\leq 1, \quad {\rm {i.e.}}  \quad \frac{c}{\|x\|_{\ell_2}
+\|\bar{x}\|_{\ell_2}}\leq 1-\frac{\eta|x_i|}{\|x\|_{\ell_2}
+\|\bar{x}\|_{\ell_2}}.
\end{equation}
Multiplying $\alpha|x_i|$ to \eqref{addeq1}, we have
\begin{equation}\label{sparsity7}
\alpha c\frac{|x_i|}{\|x\|_{\ell_2}
+\|\bar{x}\|_{\ell_2}}\leq \alpha|x_i|-\beta \frac{|x_i|^2}{\|x\|_{\ell_2}
+\|\bar{x}\|_{\ell_2}}.
\end{equation}
Denote
\begin{align*}
\displaystyle K_i:=&\frac{(\|x\|_{\ell_2}
+\|\bar{x}\|_{\ell_2})\left(\frac{1}{2}x_i\|F'(x)\|_{L(\ell_2,Y)}^2+\frac{\gamma^2}{8}x_i^3+\frac{1}{2}\gamma x_i|x_i|\|F'(x)\|_{L(\ell_2,Y)}\right)}{c\alpha}\\
&+\displaystyle \frac{(\|x\|_{\ell_2}
+\|\bar{x}\|_{\ell_2})\left(-\langle F'(x)^{*}(F(x)-y^{\delta}), e_i\rangle +\frac{\gamma}{2}x_i\|F(x)-y^{\delta}\|_Y\right)}{c\alpha}.
\end{align*}
Then a combination of (\ref{sparsity6}) and (\ref{sparsity7}) implies that
\[ K_ix_i\geq|x_i|,\quad i\in \mathbb{N}.\]
By \eqref{sparsity0}, $\|F(x)\|_Y$ is finite. In addition, Lipschitz continuity of $F'$ on $B_{\delta}(x)$ implies that $\|F'(x)\|_{L(\ell_2,Y)}$ is finite.
Since $x$, $\bar{x}\in \ell_2$, $F'(x)^{*}(F(x)-y^{\delta})\in \ell_2$, we have $K_i\to0$ as $i\to\infty$, and this
implies that $\Lambda:=\{i\in \mathbb{N}\mid|K_i|\geq 1\}$ is finite. It is obvious that $x_i=0$ whenever
$i\notin \Lambda$.  This proves the proposition.
\end{proof}


\section{Convergence rate of the regularized solutions}\label{sec3}

We consider convergence rate for the case $\alpha>\beta\geq0$ in this section. For this purpose, we need to impose a restriction on the smoothness of $x^{\dag}$. Meanwhile, we impose two commonly adopted conditions on the nonlinearity of $F$, and derive two corresponding inequalities. Then we get convergence rates $O(\delta^{\frac{1}{2}})$ and $O(\delta)$ in the $\ell_2$-norm based on the two inequalities, respectively.


\subsection{Convergence rate $O(\delta^{\frac{1}{2}})$}

\begin{assumption}\label{assumption2}
Let $x^{\dagger}\neq 0$ be an
$\mathcal R_{\eta}$-minimizing solution of the problem \eqref{equ1_1} that is sparse. Assume that

\rm {(i)}\quad $F$ is continuously Fr\'{e}chet differentiable. For every $i\in I(x^{\dag})$, there exists
$\omega_i\in D(F'(x^{\dag})^{*})$ such that
\begin{equation}\label{scondition1}
e_i=F'(x^{\dag})^{*}\omega_i,
\end{equation}
where $I(x^{\dag})$ is defined in \eqref{equ2_1a}.

\rm {(ii)}\quad  There exists $\gamma> 0$ such that
\begin{equation}\label{Lipschitz}
\|F'(x)-F'(x^{\dag})\|_{L(\ell_2, Y)}\leq \gamma\|x-x^{\dag}\|_{\ell_2}
\end{equation}
for any $x\in\ell_2$ in a sufficiently large ball around $x^{\dag}$.
\end{assumption}

Assumption \ref{assumption2} (i) and other analogous conditions were introduced in \cite{BL09,G09}. Actually, Assumption \ref{assumption2} (i) is a source condition which imposes the smoothness on the solution $x^{\dag}$. Assumption \ref{assumption2} (ii) is a restriction on $F$ which has two-fold meaning. One is to impose nonlinearity condition on $F$. Another more crucial effect is to estimate the term $\langle F'(x^{\dag})(x-x^{\dag}),\omega_i\rangle$, where $\omega_i$ are the same as that in (\ref{scondition1}). Many authors pointed out that the restrictions on the nonlinearity of $F$ coupled with source conditions prove to be a powerful tool to obtain convergence rates in regularization (\cite{GHS08,HKPS07,SKHK2012}). There are several ways to choose the restrictions on the nonlinearity of $F$. A commonly adopted restriction is (\ref{Lipschitz}), i.e. $F'$ is Lipschitz continuous (\cite{EHN1996,JM12}).

\begin{remark}\label{remark2}
Note that \eqref{Lipschitz} implies
\begin{equation}\label{Taylor}
\|F(x)-F(x^{\dag})-F'(x^{\dag})(x-x^{\dag})\|_{Y}\leq \frac{\gamma}{2}\|x-x^{\dag}\|_{\ell_2}^2
\end{equation}
from a Taylor approximation of $F$. Thus, with the triangle inequality, we obtain
\begin{equation}\label{equ2_21}
\|F'(x^{\dag})(x-x^{\dag})\|_Y\leq \frac{\gamma}{2}\|x-x^{\dag}\|_{\ell_2}^2+\|F(x)-F(x^{\dag})\|_Y
\end{equation}
which can be used to give an upper bound of the term $\langle F'(x^{\dag})(x-x^{\dag}),\omega_i\rangle$.
\end{remark}

\begin{lemma}\label{lemmaadd11} If Assumption \ref{assumption2} holds, then there exists $\displaystyle\omega^{\dag}\in Y$ such that $\displaystyle x^{\dag}=F'(x^{\dag})^*\omega^{\dag}$.
\end{lemma}

This result is verified easily by setting $\omega^{\dag}=\sum\limits_{i\in I(x^{\dag})}x_i^{\dag}\omega_i$.

Next we derive an inequality needed in the proof of the convergence rate. By Lemma \ref{lemma2_0} (i), for any $M>0$, there exists $M_1>0$ such that $ \mathcal R_{\alpha,\beta}(x)\leq M$ for $x\in \ell_2$ implies $\|x\|_{\ell_{2}}\leq M_1$. We further denote
\begin{align}\label{equadd4}
c_1&=M_1+\|x^{\dag}\|_{\ell_{2}},\\
c_2&=\left(1+\frac{c_{1}}{\|x^{\dag}\|_{\ell_{2}}}\right)|I(x^{\dag})|\max\limits_{i\in I(x^{\dag})}\|\omega_{i}\|_Y+\frac{2\|\omega^{\dag}\|_Y}{\|x^{\dag}\|_{\ell_{2}}},\\
c_3&=\frac{2\|\omega^{\dag}\|_Y}{\|x^{\dag}\|_{\ell_{2}}}.
\end{align}

\begin{lemma}\label{lemma4} Let $M>0$ be given and define $c_1$, $c_2$ and $c_3$ by (3.5)--(3.7). Under Assumption \ref{assumption2}, if $\Gamma:=\frac{\gamma c_1(c_2\alpha-c_3\beta)}{2(\alpha-\beta)}<1$, then
\begin{displaymath}
 \begin{array}{llc}
\displaystyle\|x-x^{\dag}\|^{2}_{\ell_{2}}\leq \frac{1}{\left(1-\Gamma\right)}\left[\frac{c_{1}}{\alpha-\beta}(\mathcal{R}_{\alpha,\beta}(x)-\mathcal{R}_{\alpha,\beta}(x^{\dag}))\displaystyle +\frac{c_{1}}{\alpha-\beta}(c_{2}\alpha-c_{3}\beta)\|F(x)-F(x^{\dag})\|_Y\right]
 \end{array}
\end{displaymath}
for any $x\in\ell_2$ with $ \mathcal R_{\alpha,\beta}(x)\leq M$.
\end{lemma}
\begin{proof}
By the definition of $\displaystyle \mathcal R_{\alpha,\beta}(x)$ in (\ref{equ1_3}), it is clear that
\begin{align}
\displaystyle \mathcal R_{\alpha,\beta}(x)&=\alpha\|x\|_{\ell_{1}}-\beta\|x\|_{\ell_{2}}\displaystyle=\alpha(\|x\|_{\ell_{1}}-\|x\|_{\ell_{2}})+(\alpha-\beta)\|x\|_{\ell_{2}}\nonumber\\&\displaystyle=\alpha\mathcal{K}(x)+(\alpha-\beta)\|x\|_{\ell_{2}}\nonumber,
 \end{align}
where $ \mathcal{K}(x):=\|x\|_{\ell_{1}}-\|x\|_{\ell_{2}}$. We see that
\begin{align}
\displaystyle \mathcal{R}_{\alpha,\beta}(x)-\mathcal{R}_{\alpha,\beta}(x^{\dag})&=\alpha[\mathcal{K}(x)-\mathcal{K}(x^{\dag})]+(\alpha-\beta)(\|x\|_{\ell_{2}}-\|x^{\dag}\|_{\ell_{2}})\nonumber\\
&\displaystyle=\alpha[\mathcal{K}(x)-\mathcal{K}(x^{\dag})]+(\alpha-\beta)\frac{\|x\|^{2}_{\ell_{2}}-\|x^{\dag}\|^{2}_{\ell_{2}}}{\|x\|_{\ell_{2}}+\|x^{\dag}\|_{\ell_{2}}}\nonumber\\
&\displaystyle=\alpha[\mathcal{K}(x)-\mathcal{K}(x^{\dag})]+(\alpha-\beta)\frac{\|x-x^{\dag}\|^{2}_{\ell_{2}}+2\langle \displaystyle x^{\dag},x-x^{\dag}\rangle}{\|x\|_{\ell_2}+\|x^{\dag}\|_{\ell_{2}}}.\label{lem14equ1}
 \end{align}
From the definition of $\mathcal{K}(x)$, we have
\begin{align}
\displaystyle \mathcal{K}(x)-\mathcal{K}(x^{\dag})=\|x\|_{\ell_{1}}-\|x\|_{\ell_{2}}-\|x^{\dag}\|_{\ell_{1}}+\|x^{\dag}\|_{\ell_{2}}.\nonumber
\end{align}
With the definition of index set $\displaystyle I(x^{\dag})$ in (\ref{equ2_1a}), we obtain that
\begin{align}
\displaystyle \|x\|_{\ell_{1}}=\sum\limits_{i\in I(x^{\dag})}|x_{i}|+\sum\limits_{i\notin I(x^{\dag})}|x_{i}|,\nonumber
\end{align}
\begin{align}
\displaystyle -\|x\|_{\ell_{2}}\geq -\left(\sum\limits_{i\in I(x^{\dag})}|x_{i}|^{2}\right)^{\frac{1}{2}}-\left(\sum\limits_{i\notin I(x^{\dag})}|x_{i}|^{2}\right)^{\frac{1}{2}}\geq-\left(\sum\limits_{i\in I(x^{\dag})}|x_{i}|^{2}\right)^{\frac{1}{2}}-\sum\limits_{i\notin I(x^{\dag})}|x_{i}|.\nonumber
\end{align}
Then,
\begin{align}
\displaystyle \mathcal{K}(x)-\mathcal{K}(x^{\dag})\geq\sum\limits_{i\in I(x^{\dag})}\left(|x_{i}|-|x^{\dag}_{i}|\right)-\left[\left(\sum\limits_{i\in I(x^{\dag})}|x_{i}|^{2}\right)^{\frac{1}{2}}-\left(\sum\limits_{i\in I(x^{\dag})}|x^{\dag}_{i}|^{2}\right)^{\frac{1}{2}}\right],\nonumber
\end{align}
which is rewritten as
\begin{equation}\label{equ22}
 \begin{array}{llc}
\displaystyle \mathcal{K}(x)-\mathcal{K}(x^{\dag})\geq\sum\limits_{i\in I(x^{\dag})}(|x_{i}|-|x^{\dag}_{i}|)-\frac{\sum\limits_{i\in I(x^{\dag})}(|x_{i}|^2-|x^{\dag}_{i}|^2)}{\left(\sum\limits_{i\in I(x^{\dag})}|x_{i}|^2\right)^{\frac{1}{2}}+\left(\sum\limits_{i\in I(x^{\dag})}|x_{i}^{\dag}|^2\right)^\frac{1}{2}}.
 \end{array}
 \end{equation}
By the definition of $M_1$, we have
\begin{align}
\displaystyle 0<\|x^{\dag}\|_{\ell_{2}}=\left(\sum\limits_{i\in I(x^{\dag})}|x^{\dag}_{i}|^{2}\right)^{\frac{1}{2}}\leq\left(\sum\limits_{i\in I(x^{\dag})} |x_{i}|^{2}\right)^{\frac{1}{2}}+
\left(\sum\limits_{i\in I(x^{\dag})}|x_{i}^{\dag}|^{2}\right)^{\frac{1}{2}}\leq M_{1}+\|x^{\dag}\|_{\ell_{2}}\nonumber
\end{align}
and
\begin{align}
\displaystyle |x_{i}|+|x^{\dag}_{i}|\leq M_{1}+\|x^{\dag}\|_{\ell_{2}}.\nonumber
\end{align}
 Then it follows from (\ref{equ22}) that
\begin{displaymath}
 \begin{array}{llc}
\displaystyle \mathcal{K}(x)-\mathcal{K}(x^{\dag})\geq-\sum\limits_{i\in I(x^{\dag})}|x_{i}-x^{\dag}_{i}|-\frac{M_{1}+\|x^{\dag}\|_{\ell_{2}}}{\|x^{\dag}\|_{\ell_{2}}}\sum\limits_{i\in I(x^{\dag})}|x_{i}-x^{\dag}_{i}|,
 \end{array}
 \end{displaymath}
i.e.
\begin{align}
\displaystyle \mathcal{K}(x)-\mathcal{K}(x^{\dag})\geq-\left(2+\frac{M_1}{\|x^{\dag}\|_{\ell_2}}
\right)\sum\limits_{i\in I(x^{\dag})}|x_{i}-x^{\dag}_{i}|\label{equ23add1}.
 \end{align}
A combination of (\ref{lem14equ1}) and (\ref{equ23add1}) implies that
\begin{align}
\displaystyle(\alpha-\beta)\frac{\|x-x^{\dag}\|^{2}_{\ell_{2}}}{\|x\|_{\ell_{2}}+\|x^{\dag}\|_{\ell_{2}}}
&\leq \mathcal{R}_{\alpha,\beta}(x)-\mathcal{R}_{\alpha,\beta}(x^{\dag})+\alpha\left(2+\frac{M_1}{\|x^{\dag}\|_{\ell_{2}}}\right)\sum\limits_{i\in I(x^{\dag})}|x_{i}-x^{\dag}_{i}|\nonumber\\
& \quad{}-2(\alpha-\beta)\frac{\langle x^{\dag},x-x^{\dag}\rangle}{\|x\|_{\ell_{2}}+\|x^{\dag}\|_{\ell_{2}}}.\label{equadd20}
 \end{align}
Since $\displaystyle\|x\|_{\ell_{2}}\leq M_{1}$, by \eqref{equadd20}, we see that
\begin{align}
\displaystyle\frac{(\alpha-\beta)}{M_{1}+\|x^{\dag}\|_{\ell_{2}}}\|x-x^{\dag}\|^{2}_{\ell_{2}}&\leq \mathcal{R}_{\alpha,\beta}(x)-\mathcal{R}_{\alpha,\beta}(x^{\dag})+\alpha\left(2+\frac{M_1}{\|x^{\dag}\|_{\ell_2}}\right)\sum\limits_{i\in I(x^{\dag})}|x_{i}-x^{\dag}_{i}|\nonumber\\
& \quad{} -2(\alpha-\beta)\frac{\langle x^{\dag},x-x^{\dag}\rangle}{\|x\|_{\ell_{2}}+\|x^{\dag}\|_{\ell_{2}}}.\label{equ38}
 \end{align}
By Assumption \ref{assumption2} (i), we have
\begin{displaymath}
 \begin{array}{llc}
\displaystyle|x_{i}-x^{\dag}_{i}|=|\langle e_{i},x-x^{\dag}\rangle|=|\langle \omega_{i},F'(x^{\dag})(x-x^{\dag})\rangle|\leq\max\limits_{i\in I(x^{\dag})}\|\omega_{i}\|_Y\|F'(x^{\dag})(x-x^{\dag})\|_Y.
 \end{array}
 \end{displaymath}
Hence
\begin{equation}\label{equ40}
 \begin{array}{llc}
\displaystyle \sum\limits_{i\in I(x^{\dag})}|x_{i}-x_{i}^{\dag}|\leq|I(x^{\dag})|\max\limits_{i\in I(x^{\dag})}\|\omega_{i}\|_Y\|F'(x^{\dag})(x-x^{\dag})\|_Y,\\
 \end{array}
 \end{equation}
where $\displaystyle |I(x^{\dag})|$ denotes the size of the index set $\displaystyle I(x^{\dag})$.
On the other hand, by Lemma \ref{lemmaadd11}, we see that
\begin{equation}\label{equ41}
\frac{|\langle x^{\dag},x-x^{\dag}\rangle|}{\|x\|_{\ell_{2}}+\|x^{\dag}\|_{\ell_{2}}}=\frac{|\langle \omega^{\dag},F'(x^{\dag})(x-x^{\dag})\rangle|}{\|x\|_{\ell_{2}}+\|x^{\dag}\|_{\ell_{2}}}\leq\frac{\|\omega^{\dag}\|_Y\|F'(x^{\dag})(x-x^{\dag})\|_Y}{\|x^{\dag}\|_{\ell_{2}}}.
\end{equation}
A combination of (\ref{equ38}), (\ref{equ40}) (\ref{equ41}) and (\ref{equ2_21}) implies that
\begin{align}
\displaystyle & \quad\frac{\alpha-\beta}{M_{1}+\|x^{\dag}\|_{\ell_{2}}}\|x-x^{\dag}\|^{2}_{\ell_{2}} \nonumber \\ &\leq \mathcal{R}_{\alpha,\beta}(x)-\mathcal{R}_{\alpha,\beta}(x^{\dag})
\displaystyle \nonumber\\
&\displaystyle \quad+\left(\alpha\left(2+\frac{M_1}{\|x^{\dag}\|_{\ell_2}}\right)|I(x^{\dag})|
\max\limits_{i\in I(x^{\dag})}\|\omega_{i}\|_Y+2(\alpha-\beta)\frac{\|\omega^{\dag}\|_Y}{\|x^{\dag}\|_{\ell_{2}}}\right)(\|F(x)-F(x^{\dag})\|_Y+\frac{\gamma}{2}\|x-x^{\dag}\|^2),\nonumber
\end{align}
i.e.
\begin{displaymath}
\displaystyle\|x-x^{\dag}\|^{2}_{\ell_{2}}\leq \frac{1}{\left(1-\frac{\gamma c_1(c_2\alpha-c_3\beta)}{2(\alpha-\beta)}\right)}     \left[\frac{c_{1}}{\alpha-\beta}(\mathcal{R}_{\alpha,\beta}(x)-\mathcal{R}_{\alpha,\beta}(x^{\dag}))\displaystyle +\frac{c_{1}}{\alpha-\beta}(c_{2}\alpha-c_{3}\beta)\|F(x)-F(x^{\dag})\|_Y\right], \nonumber
\end{displaymath}
where $c_1$, $c_{2}$ and $c_{3}$ are defined by \eqref{equadd4}.
\end{proof}

\begin{theorem}\label{theorem1}
Suppose Assumption \ref{assumption2} holds. Let $x^{\delta}_{\alpha,\beta}$ be defined by \eqref{equ2_1}, and let the constants $c_1>0, c_2>c_3>0$ be as in Lemma \ref{lemma4}. Assume $\Gamma:=\frac{\gamma c_1(c_2\alpha-c_3\beta)}{2(\alpha-\beta)}<1$.

\begin{subequations}
1. If $q=1$ and $c_2\alpha-c_3\beta<1$, then
\begin{align}
\displaystyle \|x^{\delta}_{\alpha,\beta}-x^{\dag}\|^{2}_{\ell_{2}}\leq \displaystyle\frac{c_1[1+(c_{2}\alpha-c_{3}\beta)]\delta}{(\alpha-\beta)(1-\Gamma)},\qquad
\displaystyle \|F(x^{\delta}_{\alpha,\beta})-y^{\delta}\|_Y\leq \frac{[1+(c_{2}\alpha-c_{3}\beta)]\delta}{1-(c_{2}\alpha-c_{3}\beta)}.\label{equ21a}
 \end{align}

2. If $q>1$, then
\begin{align}
&\displaystyle\|x^{\delta}_{\alpha,\beta}-x^{\dag}\|^{2}_{\ell_2}\leq\frac{c_{1}}{(\alpha-\beta)(1-\Gamma)}\left[\frac{\delta^{q}}{q}+(c_{2}\alpha-c_{3}\beta)\delta+\frac{(q-1)2^{\frac{1}{q-1}}(c_{2}\alpha-c_{3}\beta)^{\frac{q}{q-1}}}{q}\right],\nonumber
\\\displaystyle&\|F(x^{\delta}_{\alpha,\beta})-y^{\delta}\|_Y^q\leq q\left[\frac{\delta^{q}}{q}+(c_{2}\alpha-c_{3}\beta)\delta+\frac{(q-1)2^{\frac{1}{q-1}}(c_{2}\alpha-c_{3}\beta)^{\frac{q}{q-1}}}{q}\right].\label{equ21b}
 \end{align}
\end{subequations}
\end{theorem}
\begin{proof} By the definition of $\displaystyle x^{\delta}_{\alpha,\beta}$, it is clear that
\begin{displaymath}
 \begin{array}{llc}
\displaystyle\frac{1}{q}\|F(x^{\delta}_{\alpha,\beta})-y^{\delta}\|_Y^{q}+\mathcal{R}(x^{\delta}_{\alpha,\beta})\leq\frac{1}{q}\|F(x^{\dag})-y^{\delta}\|_Y^{q}+\mathcal{R}(x^{\dag}),
 \end{array}
 \end{displaymath}
i.e.
\begin{align}
\displaystyle\frac{\delta^{q}}{q}&\displaystyle\geq \mathcal{R}(x^{\delta}_{\alpha,\beta})-\mathcal{R}(x^{\dag})+\frac{1}{q}\|F(x^{\delta}_{\alpha,\beta})-y^{\delta}\|_Y^{q}.\nonumber
\end{align}
Then $\displaystyle \mathcal{R}(x^{\delta}_{\alpha,\beta})$ is bounded. Applying Lemma \ref{lemma4}, we see that
\begin{align}
\displaystyle\frac{\delta^{q}}{q}&\displaystyle\geq\frac{(\alpha-\beta)(1-\Gamma)}{c_{1}}\|x^{\delta}_{\alpha,\beta}-x^{\dag}\|^{2}_{\ell_2}-(c_{2}\alpha-c_{3}\beta)\|F(x^{\delta}_{\alpha,\beta})-F(x^{\dag})\|_Y
+\frac{1}{q}\|F(x^{\delta}_{\alpha,\beta})-y^{\delta}\|_Y^{q}\nonumber\\&
\displaystyle\geq\frac{(\alpha-\beta)(1-\Gamma)}{c_{1}}\|x^{\delta}_{\alpha,\beta}-x^{\dag}\|^{2}_{\ell_2}-(c_{2}\alpha-c_{3}\beta)\|F(x^{\delta}_{\alpha,\beta})-y^{\delta}\|_Y\nonumber\\&\displaystyle
\displaystyle\quad-(c_{2}\alpha-c_{3}\beta)\delta+\frac{1}{q}\|F(x^{\delta}_{\alpha,\beta})-y^{\delta}\|_Y^{q}.\label{equ48}
 \end{align}
So if $\displaystyle q=1$ and $c_2\alpha-c_3\beta<1$, then (\ref{equ21a}) holds.

If $\displaystyle q>1$, we apply Young's inequality $ ab\leq\frac{a^{q}}{q}+\frac{b^{q^{*}}}{q^{*}}$ for $a, b\geq0$ and $q^*>1$ defined by $\frac{1}{q}+\frac{1}{q^{*}}=1$. We have
\begin{align}
\displaystyle(c_{2}\alpha-c_{3}\beta)\|F(x^{\delta}_{\alpha,\beta})-y^{\delta}\|_Y&\displaystyle=2^{\frac{1}{q}}(c_{2}\alpha-c_{3}\beta)2^{-\frac{1}{q}}\|F(x^{\delta}_{\alpha,\beta})-y^{\delta}\|_Y\nonumber
\\&\displaystyle\leq\frac{1}{2q}\|F(x^{\delta}_{\alpha,\beta})-y^{\delta}\|_Y^{q}+\frac{(q-1)2^{\frac{1}{q-1}}\left(c_{2}\alpha-c_{3}\beta\right)^{\frac{q}{q-1}}}{q}.\label{equ49}
 \end{align}
A combination with (\ref{equ48}) and (\ref{equ49}) implies (\ref{equ21b}).
\end{proof}

\begin{remark}{\rm (A-priori estimation)}
Let $\displaystyle \beta=\eta\alpha$ be a fixed constant. For the case $q>1$, if $\alpha\sim\delta^{q-1}$, then
$\|x^{\delta}_{\alpha,\beta}-x^{\dag}\|_{\ell_2}\leq c\delta^{\frac{1}{2}}$ for some constant $\displaystyle c>0$.
For the particular case $q=1$, if $\alpha\sim\delta^{1-\epsilon}$ $(0<\epsilon<1)$, then
$\|x^{\delta}_{\alpha,\beta}-x^{\dag}\|_{\ell_2}\leq c\delta^{\frac{\epsilon}{2}}$ for some constant $\displaystyle c>0$.
\end{remark}

Note that due to the presence of the term $\frac{\gamma}{2}\|x-x^{\dag}\|_{\ell_2}^2$ in the estimation (\ref{equ2_21}),
we need an additional condition to obtain the convergence rate, i.e.\ $\gamma>0$ must be small enough such that
$\Gamma<1$. This additional condition is similar to the condition $\gamma\, \|\omega\|<1$ in the classical quadratic regularization (\cite{EHN1996}).

\begin{theorem} {\rm (Discrepancy principle)}
Keep the assumptions of Lemma \ref{lemma4} and let $\displaystyle x^{\delta}_{\alpha,\beta}$ be
defined by \eqref{equ2_1}, where the parameters $\alpha$ and $\beta$ $(\beta=\eta\alpha)$ are defined
via the discrepancy principle
\[\delta\leq\|F(x_{\alpha,\beta}^{\delta})-y^{\delta}\|_Y\leq \tau\delta~~(\tau\geq 1).\]
Then
\[\|x_{\alpha,\beta}^{\delta}-x^{\dag}\|_{\ell_2}\le\sqrt{\frac{c_1(c_2-c_3\eta)(\tau+1)\delta}{(1-\eta)(1-\Gamma)}}.\]
\end{theorem}  
\begin{proof}
By the definition of $x_{\alpha,\beta}^{\delta}$, $\alpha$ and $\beta$, we see that
\begin{equation}  
\displaystyle \frac{1}{q}{\delta}^q+\mathcal{R}_{\alpha,\beta}(x_{\alpha,\beta}^{\delta})\displaystyle\leq \frac{1}{q}\|F(x_{\alpha,\beta}^{\delta})-y^{\delta}\|_Y^q+\mathcal{R}_{\alpha,\beta}(x_{\alpha,\beta}^{\delta})\leq \frac{1}{q}\|F(x^{\dag})-y^{\delta}\|_Y^q+\mathcal{R}_{\alpha,\beta}(x^{\dag}).
 \end{equation}
Hence $\mathcal{R}_{\alpha,\beta}(x_{\alpha,\beta}^{\delta})\leq \mathcal{R}_{\alpha,\beta}(x^{\dag})$. It follows from Lemma \ref{lemma4} that
\begin{align}
\displaystyle 0\geq \mathcal{R}_{\alpha,\beta}(x_{\alpha,\beta}^{\delta})-\mathcal{R}_{\alpha,\beta}(x^{\dag})&\displaystyle\geq
\frac{(\alpha-\beta)(1-\Gamma)}{c_{1}}\|x^{\delta}_{\alpha,\beta}-x^{\dag}\|^{2}_{\ell_2}-(c_{2}\alpha-c_{3}\beta)\|F(x^{\delta}_{\alpha,\beta})-F(x^{\dag})\|_Y
\displaystyle \nonumber\\& \displaystyle\geq \frac{(\alpha-\beta)(1-\Gamma)}{c_{1}}\|x^{\delta}_{\alpha,\beta}-x^{\dag}\|^{2}_{\ell_2}-(c_{2}\alpha-c_{3}\beta)(\tau+1)\delta.
\end{align}
Then
\[ \|x_{\alpha,\beta}^{\delta}-x^{\dag}\|_{\ell_{2}}^2\leq\frac{c_1(c_{2}\alpha-c_{3}\beta)(\tau+1)\delta}{(\alpha-\beta)(1-\Gamma)}.\]
The theorem is proven with $\beta=\eta\alpha$.
\end{proof}


\subsection{Convergence rate $O(\delta)$}

In \cite[p.\ 6]{KNS2012}, it is pointed out that for ill-posed problems, (\ref{Taylor}) carries too little information
about the local behaviour of $F$ around $x^{\dag}$ to draw conclusions about convergence, since the left hand side
of (\ref{Taylor}) can be much smaller than the right hand
side for certain pairs of points $x$ and $x^{\dag}$, no matter how close to each other they are. Therefore, several researchers adopted
\begin{equation}\label{equ2_32}
\|F(x)-F(x^{\dag})-F'(x^{\dag})(x-x^{\dag})\|_Y\leq \gamma\|F(x)-F(x^{\dag})\|_Y, \quad 0<\gamma<\frac{1}{2}
\end{equation}
as the condition on the nonlinearity of $F$, see \cite[pp.\ 278--279]{EHN1996}, \cite[p.\ 6]{KNS2012},
\cite[pp.\ 69--70]{SKHK2012}.

\begin{assumption}\label{assumption3}
Let $x^{\dagger}\neq 0$ be an $\mathcal R_{\eta}$-minimizing solution of the problem $F(x)=y$ that is sparse. We further assume that

\rm{(i)}\quad $F$ is Fr\'{e}chet differentiable at $x^{\dag}$. There exists an $\omega_i\in D(F'(x^{\dag})^{*})$ such that
\begin{equation}\label{scondition}
e_i=F'(x^{\dag})^{*}\omega_i\quad for\ every\quad i\in I(x^{\dag}),
\end{equation}
where $I(x^{\dag})$ is defined in \eqref{equ2_1a}.

\rm{(ii)}\quad There exist constants $0<\gamma<\frac{1}{2}$ and $\delta>0$ such that
\begin{equation}\label{nonlicondition}
\displaystyle \|F(x)-F(x^{\dag})-F'(x^{\dag})(x-x^{\dag})\|_Y\leq \gamma\|F(x)-F(x^{\dag})\|_Y
 \end{equation}
for any $x\in\ell_2\cap B_{\delta}(x^{\dag})$, where $B_{\delta}(x^{\dag}):=\{x \mid \|x-x^{\dag}\|_{\ell_2}\leq \delta\}$.
\end{assumption}

\begin{remark}\label{remark3}
With the triangle inequality, it follows from \eqref{nonlicondition} that
\begin{equation}\label{equ2_33}
\frac{1}{1+\gamma}\|F'(x^{\dag})(x-x^{\dag})\|_Y\le\|F(x)-F(x^{\dag})\|_Y\le\frac{1}{1-\gamma}\|F'(x^{\dag})(x-x^{\dag})\|_Y
\end{equation}
which is an estimate for $\|F'(x^{\dag})(x-x^{\dag})\|_Y$.  Actually, a more direct restriction
\[\|F'(x^{\dag})(x-x^{\dag})\|_Y\leq(1+\gamma)\|F(x)-F(x^{\dag})\|_Y\]
has been adopted by several researchers. This assumption immediately leads to a bound of
the critical inner product $\langle F'(x^{\dag})(x-x^{\dag}),\omega_i\rangle$.
\end{remark}

Next, we derive an inequality from the restriction (\ref{nonlicondition}). The linear convergence rate $O(\delta)$
follows from the inequality directly.

\begin{lemma}\label{lemma5} Let Assumption \ref{assumption3} hold and $\mathcal R_{\alpha,\beta}(x)\leq M$
for a given $M>0$. Then there exist constants $c_{4}>c_{5}$ such that
\begin{equation}\label{equ4}
(\alpha-\beta)\|x-x^{\dag}\|_{\ell_{1}}\leq \mathcal{R}_{\alpha,\beta}(x)-\mathcal{R}_{\alpha,\beta}(x^{\dag})+(c_{4}\alpha-c_{5}\beta)\|F(x)-F(x^{\dag})\|_Y.
\end{equation}
\end{lemma}
\begin{proof}
By the definition of $I(x^{\dag})$, we have
\[\displaystyle(\alpha-\beta)\|x-x^{\dag}\|_{\ell_{1}}=(\alpha-\beta)\left(\sum\limits_{i\in I(x^{\dag})}|x_{i}-x^{\dag}_{i}|+\sum\limits_{i\notin I(x^{\dag})}|x_{i}|\right).\]
Then,
\begin{align}
&\displaystyle(\alpha-\beta)\|x-x^{\dag}\|_{\ell_{1}}-(\mathcal{R}_{\alpha,\beta}(x)-\mathcal{R}_{\alpha,\beta}(x^{\dag}))\nonumber\\\displaystyle=&-\alpha\sum\limits_{i\in I(x^{\dag})}\left(|x_{i}|-|x_i^{\dag}|\right)+(\alpha-\beta)\sum\limits_{i\in I(x^{\dag})}|x_{i}-x^{\dag}_{i}|+\beta(T_1-T_2),\label{equ27}
\end{align}
where
\begin{align*}
T_1&=\left(\sum\limits_i|x_{i}|^{2}\right)^{\frac{1}{2}}-\left(\sum\limits_{i\notin I(x^{\dag})}|x_{i}|^{2}\right)^{\frac{1}{2}}-\left(\sum\limits_{i\in I(x^{\dag})}|x^{\dag}_{i}|^{2}\right)^{\frac{1}{2}},\\
T_2&=\sum\limits_{i\notin I(x^{\dag})}|x_{i}|-\left(\sum\limits_{i\notin I(x^{\dag})}|x_{i}|^{2}\right)^{\frac{1}{2}}.
\end{align*}
Observe that $T_2\geq 0$. Since
\[\left(\sum\limits_i|x_{i}|^{2}\right)^{\frac{1}{2}}\leq \left(\sum\limits_{i\in I(x^{\dag})}|x_{i}|^{2}\right)^{\frac{1}{2}}
+\left(\sum\limits_{i\notin I(x^{\dag})}|x_{i}|^{2}\right)^{\frac{1}{2}},\]
we see that
\begin{align}
T_1\leq T_3:=\left(\sum\limits_{i\in I(x^{\dag})}|x_{i}|^{2}\right)^{\frac{1}{2}}-\left(\sum\limits_{i\in I(x^{\dag})}|x^{\dag}_{i}|^{2}\right)^{\frac{1}{2}}.
\end{align}
Thus, from (\ref{equ27}),
\begin{align}
\displaystyle (\alpha-\beta)\|x-x^{\dag}\|_{\ell_{1}}&\leq \mathcal{R}_{\alpha,\beta}(x)-\mathcal{R}_{\alpha,\beta}(x^{\dag})+\alpha\sum\limits_{i\in I(x^{\dag})}|x_{i}-x^{\dag}_{i}|\nonumber\\&
\quad+(\alpha-\beta)\sum\limits_{i\in I(x^{\dag})}|x_{i}-x^{\dag}_{i}|+\beta T_3.\label{equ55}
 \end{align}
Let the constant $M_1$ be as in the proof of Lemma \ref{lemma4}. Then
\begin{align*}
& |x_{i}|+|x^{\dag}_{i}|\leq M_1+\|x^{\dag}\|_{\ell_{2}},\\
& 0<\|x^{\dag}\|_{\ell_{2}}\leq\left(\sum\limits_{i\in I(x^{\dag})}|x_{i}|^{2}\right)^{\frac{1}{2}}+\left(\sum\limits_{i\in I(x^{\dag})}|x^{\dag}_{i}|^{2}\right)^{\frac{1}{2}}.
\end{align*}
Consequently
\begin{equation}\label{equ56}
T_3=\frac{\sum\limits_{i\in I(x^{\dag})}(|x_{i}|-|x^{\dag}_{i}|)(|x_{i}|+|x^{\dag}_{i}|)}{\left(\sum\limits_{i\in I(x^{\dag})}|x_{i}|^{2}\right)^{\frac{1}{2}}+\left(\sum\limits_{i\in I(x^{\dag})}|x^{\dag}_{i}|^{2}\right)^{\frac{1}{2}}}\leq\frac{M_1+\|x^{\dag}\|_{\ell_{2}}}{\|x^{\dag}\|_{\ell_{2}}}\sum\limits_{i\in I(x^{\dag})}|x_{i}-x^{\dag}_{i}|.
\end{equation}
A combination of (\ref{equ55}) and (\ref{equ56}) shows that
\begin{equation}\label{equ57}
 \begin{array}{llc}
\displaystyle (\alpha-\beta)\|x-x^{\dag}\|_{\ell_{1}}\leq \mathcal{R}_{\alpha,\beta}(x)-\mathcal{R}_{\alpha,\beta}(x^{\dag})+\left(2\alpha+\frac{M_1}{\|x^{\dag}\|_{\ell_2}}\beta\right)\sum\limits_{i\in I(x^{\dag})}|x_{i}-x^{\dag}_{i}|.
\end{array}
\end{equation}
In addition, by Assumption \ref{assumption2},
\[\displaystyle \displaystyle |x_{i}-x_{i}^{\dag}|=|\langle e_{i},x-x^{\dag}\rangle|=|\langle \omega_{i},F'(x^{\dag})(x-x^{\dag})\rangle|\leq\max\limits_{i\in I(x^{\dag})}\|\omega_{i}\|_Y\|F'(x^{\dag})(x-x^{\dag})\|_Y.\]
Hence,
\[  \displaystyle\sum\limits_{i\in I(x^{\dag})}|x_{i}-x^{\dag}_{i}|\leq|I(x^{\dag})|\max\limits_{i\in I(x^{\dag})}\|\omega_{i}\|_Y\|F'(x^{\dag})(x-x^{\dag})\|_Y,  \]
where $\displaystyle |I(x^{\dag})|$ denotes the size of the index set $\displaystyle I(x^{\dag})$. Then, by (\ref{nonlicondition}), we have
\begin{equation}\label{equ59}
\displaystyle\sum\limits_{i\in I(x^{\dag})}|x_{i}-x^{\dag}_{i}|\leq|I(x^{\dag})|\max\limits_{i\in I(x^{\dag})}\|\omega_{i}\|_Y(1+\gamma)\|F(x)-F(x^{\dag})\|_Y.
 \end{equation}
A combination of (\ref{equ57}) and (\ref{equ59}) implies that
\begin{align}
(\alpha-\beta)\|x-x^{\dag}\|_{\ell_1}&\le\mathcal{R}_{\alpha,\beta}(x)-\mathcal{R}_{\alpha,\beta}(x^{\dag})\nonumber
\\
&\quad{}+\left(2\alpha+\frac{M_1}{\|x^{\dag}\|_{\ell_2}}\beta\right)|I(x^{\dag})|\max\limits_{i\in I(x^{\dag})}\|\omega_{i}\|_Y(1+\gamma)\|F(x)-F(x^{\dag})\|_Y\label{equ2_43}
 \end{align}
i.e.
\[\displaystyle (\alpha-\beta)\|x-x^{\dag}\|_{\ell_{1}}\leq \mathcal{R}_{\alpha,\beta}(x)-\mathcal{R}_{\alpha,\beta}(x^{\dag})+(c_{4}\alpha-c_{5}\beta)\|F(x)-F(x^{\dag})\|_Y,\]
where
\[c_{4}=2|I(x^{\dag})|\max\limits_{i\in I(x^{\dag})}\|\omega_{i}\|_Y(1+\gamma),\quad
c_{5}=-\frac{M_1}{\|x^{\dag}\|_{\ell_2}}|I(x^{\dag})|\max\limits_{i\in I(x^{\dag})}\|\omega_{i}\|_Y(1+\gamma)\]
and $c_{4}\alpha-c_{5}\beta>0$. The proof is completed.
\end{proof}

\begin{theorem}\label{theorem2}
Suppose Assumption \ref{assumption2} holds. Let $x^{\delta}_{\alpha,\beta}$ be defined by
\eqref{equ2_1} and let the constants $c_4>c_5$ be as in Lemma \ref{lemma5}.

\begin{subequations}
1. If $\displaystyle q=1$ and $1-(c_{4}\alpha-c_{5}\beta)>0$, then
\begin{align}
\displaystyle \|x^{\delta}_{\alpha,\beta}-x^{\dag}\|_{\ell_{1}}\leq
\frac{1+(c_{4}\alpha-c_{5}\beta)}{(\alpha-\beta)}\delta,\quad\quad
\|F(x^{\delta}_{\alpha,\beta})-y^{\delta}\|_Y\leq \frac{1+(c_{4}\alpha-c_{5}\beta)}
{1-(c_{4}\alpha-c_{5}\beta)}\delta.\label{equ33a}
 \end{align}

2. If $\displaystyle q>1$, then
\begin{align}
\displaystyle &\|x^{\delta}_{\alpha,\beta}-x^{\dag}\|_{\ell_{1}}\leq\frac{1}{\alpha-\beta}\left[\frac{\delta^{q}}{q}+(c_{4}\alpha-c_{5}\beta)\delta+\frac{(q-1)2^{\frac{1}{q-1}}(c_{4}\alpha-c_{5}\beta)^{\frac{q}{q-1}}}{q}\right],\nonumber
\\ &\displaystyle \|F(x^{\delta}_{\alpha,\beta})-y^{\delta}\|_Y^q\leq q\left[\frac{\delta^{q}}{q}+(c_{4}\alpha-c_{5}\beta)\delta+\frac{(q-1)2^{\frac{1}{q-1}}(c_{4}\alpha-c_{5}\beta)^{\frac{q}{q-1}}}{q}\right].\label{equ33b}
 \end{align}
\end{subequations}
\end{theorem}
\begin{proof}
By the definition of $\displaystyle x^{\delta}_{\alpha,\beta}$, it is obvious that
\[\displaystyle \frac{1}{q}\|F(x^{\delta}_{\alpha,\beta})-y^{\delta}\|_Y^{q}+\mathcal R_{\alpha,\beta}(x^{\delta}_{\alpha,\beta})\leq\frac{\delta^{q}}{q}+\mathcal{R}_{\alpha,\beta}(x^{\dag}).\]
Then $\displaystyle \mathcal R_{\alpha,\beta}(x^{\delta}_{\alpha,\beta})$ is bounded. From Lemma \ref{lemma5} we see that
\begin{align}
\displaystyle\frac{\delta^{q}}{q}&\geq \mathcal R_{\alpha,\beta}(x^{\delta}_{\alpha,\beta})-\mathcal R_{\alpha,\beta}(x^{\dag})+\frac{1}{q}\|F(x^{\delta}_{\alpha,\beta})-y^{\delta}\|_Y^{q}\nonumber\\
\displaystyle&\geq(\alpha-\beta)\|x^{\delta}_{\alpha,\beta}-x^{\dag}\|_{\ell_{1}}-(c_{4}\alpha-c_{5}\beta)\|F(x^{\delta}_{\alpha,\beta})-F(x^{\dag})\|_Y+\frac{1}{q}\|F(x^{\delta}_{\alpha,\beta})-y^{\delta}\|_Y^{q}\nonumber\\
\displaystyle&\geq(\alpha-\beta)\|x^{\delta}_{\alpha,\beta}-x^{\dag}\|_{\ell_{1}}-(c_{4}\alpha-c_{5}\beta)\|F(x^{\delta}_{\alpha,\beta})-y^{\delta}\|_Y-(c_{4}\alpha-c_{5}\beta)\delta+\frac{1}{q}\|F(x^{\delta}_{\alpha,\beta})-y^{\delta}\|_Y^{q}.\label{equ64}
 \end{align}
So if $q=1$ and $1-(c_{4}\alpha-c_{5}\beta)>0$, then (\ref{equ33a}) holds. For the case $q>1$,
we apply Young's inequality $ ab\le a^q/q+b^{q^*}/q^{*}$ for $a, b\geq0$ and $q^*>1$ defined by $1/q+1/q^{*}=1$. We have
\begin{align}
(c_{4}\alpha-c_{5}\beta)\|F(x^{\delta}_{\alpha,\beta})-y^{\delta}\|_Y
&=2^{\frac{1}{q}}(c_{4}\alpha-c_{5}\beta)2^{-\frac{1}{q}}\|F(x^{\delta}_{\alpha,\beta})-y^{\delta}\|_Y\nonumber\\
&\leq\frac{1}{2q}\|F(x^{\delta}_{\alpha,\beta})-y^{\delta}\|_Y^{q}+\frac{(q-1)2^{\frac{1}{q-1}}\left(c_{4}\alpha-c_{5}\beta\right)^{\frac{q}{q-1}}}{q}.\label{equ65}
 \end{align}
A combination of (\ref{equ64}) and (\ref{equ65}) implies (\ref{equ33b}).
\end{proof}

\begin{remark} {\rm (A-priori estimation)}
Assume $\displaystyle \beta=\eta\alpha$ for a constant $\eta>0$. If $ \alpha\sim\delta^{q-1}$ with $q>1$,
then $\|x^{\delta}_{\alpha,\beta}-x^{\dag}\|_{\ell_1}\leq c\delta$ for some constant $c>0$. Obviously, then we also have
$\|x^{\delta}_{\alpha,\beta}-x^{\dag}\|_{\ell_2}\leq c\delta$. For the particular case $q=1$,
if $\alpha\sim\delta^{1-\epsilon}$ $(0<\epsilon<1)$, then
$\|x^{\delta}_{\alpha,\beta}-x^{\dag}\|_{\ell_2}\leq c\delta^{\epsilon}$ for some constant $\displaystyle c>0$.
\end{remark}

Note that we can not get the inequality (\ref{equ4}) if the restriction (\ref{nonlicondition}) is replaced by (\ref{Lipschitz}), since the term $\frac{\gamma}{2}\|x-x^{\dag}\|_{\ell_2}^2$ appears in (\ref{equ2_43}) and we can not combine the terms $\|x-x^{\dag}\|_{\ell_{1}}$ and $ \frac{\gamma}{2}\|x-x^{\dag}\|_{\ell_2}^2$. Then we can not obtain the desired convergence rate $O(\delta)$. In addition, we note that the above results on the convergence rate
hold only for the case $\alpha>\beta$. When $\alpha=\beta$, Lemmas \ref{lemma4} and \ref{lemma5} are no longer meaningful.
So the proofs of the convergence rate are invalid if $\alpha=\beta$.

\begin{theorem} {\rm (Discrepancy principle)}
Keep the assumptions of Lemma \ref{lemma5} and let $\displaystyle x^{\delta}_{\alpha,\beta}$ be
defined by \eqref{equ2_1}, where the parameters $\alpha$ and $\beta$ $(\beta=\eta\alpha)$ are chosen
via the discrepancy principle
\[\delta\leq\|F(x_{\alpha,\beta}^{\delta})-y^{\delta}\|_Y\leq \tau\delta~~(\tau\geq 1).\]
 Then
\[\|x_{\alpha,\beta}^{\delta}-x^{\dag}\|_{\ell_{2}}\leq \frac{(c_{1}-c_{2}\eta)(\tau+1)\delta}{1-\eta}.\]
\end{theorem}  
\begin{proof}
By the definition of $x_{\alpha,\beta}^{\delta}$, $\alpha$ and $\beta$, we see that
\begin{equation}  
\displaystyle \frac{1}{q}{\delta}^q+\mathcal{R}_{\alpha,\beta}(x_{\alpha,\beta}^{\delta})\displaystyle\leq \frac{1}{q}\|F(x_{\alpha,\beta}^{\delta})-y^{\delta}\|_Y^q+\mathcal{R}_{\alpha,\beta}(x_{\alpha,\beta}^{\delta})\leq \frac{1}{q}\|F(x^{\dag})-y^{\delta}\|_Y^q+\mathcal{R}_{\alpha,\beta}(x^{\dag}).
 \end{equation}
Hence $\mathcal{R}_{\alpha,\beta}(x_{\alpha,\beta}^{\delta})\leq \mathcal{R}_{\alpha,\beta}(x^{\dag})$. It follows from Lemma \ref{lemma4} that
\begin{align}
\displaystyle 0\geq \mathcal{R}_{\alpha,\beta}(x_{\alpha,\beta}^{\delta})-\mathcal{R}_{\alpha,\beta}(x^{\dag})&\displaystyle\geq (\alpha-\beta)\|x_{\alpha,\beta}^{\delta}-x^{\dag}\|_{\ell_{1}}-(c_{1}\alpha-c_{2}\beta)\|F(x_{\alpha,\beta}^{\delta})-F(x^{\dag})\|_Y
\displaystyle \nonumber\\& \displaystyle\geq (\alpha-\beta)\|x_{\alpha,\beta}^{\delta}-x^{\dag}\|_{\ell_{1}}-(c_{1}\alpha-c_{2}\beta)(\tau+1)\delta.
\end{align}
Then
\[ \|x_{\alpha,\beta}^{\delta}-x^{\dag}\|_{\ell_{2}}\leq \|x_{\alpha,\beta}^{\delta}-x^{\dag}\|_{\ell_{1}}\leq\frac{(c_{1}\alpha-c_{2}\beta)(\tau+1)\delta}{\alpha-\beta}.\]
The theorem is proven with $\beta=\eta\alpha$.
\end{proof}


\section{Computational approach}\label{sec4}

In this section we introduce and analyze a solution algorithm for the problem (\ref{equ1_2}) in the finite dimensional space $\mathbb{R}^n$. We propose an
iterative soft thresholding algorithm based on the generalized conditional gradient method (GCGM). We prove
the convergence of the algorithm and show that GCGM can be applied to the $\alpha\| x\|_{\ell_1}-\beta\| x\|_{\ell_2}$
$(\alpha\geq\beta\geq0)$ sparsity regularization for nonlinear inverse problems.


\subsection{Generalized conditional gradient method}\label{sec3.1}

In \cite{BBLM07,BLM09}, GCGM was proposed to solve a minimization problem for a functional $G(x)+\Phi(x)$ on a
Hilbert space $H$, where $G: H\rightarrow \mathbb{R}$ is continuously Fr\'echet differentiable and
$\Phi: H\rightarrow \mathbb{R}\cup \{\infty\}$ is proper, convex, lower semi-continuous and coercive.
In addition, GCGM has been applied to solve the classical sparsity regularization by setting
$ G(x)=\frac{1}{2}\|F(x)-y^{\delta}\|_Y^{2}-\frac{\lambda}{2}\|x\|_{\ell_2}^2$ and
$ \Phi(x)=\frac{\lambda}{2}\|x\|_{\ell_2}^2+\alpha\sum\limits_n w_n|\langle u, \phi_n\rangle|^p$ with $p\geq 1$, where
$\{w_n>0\}$ are the weights, $\{\phi_n\}$ is an orthonormal basis of $H$, and $F$ is a linear (or nonlinear) operator.
GCGM from \cite{BLM09} is stated in the form of Algorithm 1.

\begin{algorithm}\label{alg11}
\caption{Generalized conditional gradient method}
\begin{algorithmic}
\STATE{1: Choose $x^0\in H$ such that $\Phi(x^0)<+\infty$, and set $k=0$.}
\STATE{2: Determine a solution $z^k$ by solving \[\min\limits_{z \in H} \langle G^{\prime}(x^k),z\rangle+\Phi(z).\]}
\STATE{3: Set a step size $s^k$ as a solution of \[\min\limits_{s \in [0,1]} G(x^k+s(z^k-x^k))+\Phi(x^k+s(z^k-x^k)).\]}
\STATE{4: Put $x^{k+1}=x^k+s_k(z^k-x^k)$, and $k = k + 1$, return to Step 2.}
\end{algorithmic}
\end{algorithm}

We now consider applying GCGM to solve the problem (\ref{equ1_2}) in the finite dimensional space $\mathbb{R}^n$. In this section, we assume that the a nonlinear operator $F: \mathbb{R}^n\rightarrow \mathbb{R}^m$ is continuously Fr\'echet differentiable and is bounded on bounded sets. For simplicity, we only consider the case $q=2$ in (\ref{equ1_2}). Since the term $\displaystyle \alpha\| x\|_{\ell_1}-\beta\| x\|_{\ell_2}$ is not convex, a property required by GCGM, we rewrite $\mathcal{J}_{\alpha,\beta}^{\delta}(x)$ in (\ref{equ1_2}) in the finite dimensional space $\mathbb{R}^n$ as
\begin{equation}\label{equ4_1}
\mathcal{J}_{\alpha,\beta}^{\delta}(x)=G(x)+\Phi(x),
\end{equation}
where
\[ \displaystyle G(x)=\frac{1}{2}\| F(x)-y^{\delta}\|_{\ell_2}^2-\Theta(x),\quad\Phi(x)=\Theta(x)+\alpha\|x\|_{\ell_1}-\beta\| x\|_{\ell_2},\]
and $\Theta(x)=\frac{\lambda}{2}\|x\|_{\ell_2}^2+\beta\|x\|_{\ell_2}$, $\lambda>0$. Thus, the problem (\ref{equ1_2}) can be expressed as
\begin{equation}\label{equ3_2}
\min\limits_x \mathcal{J}_{\alpha,\beta}^{\delta}(x)=G(x)+\Phi(x).
\end{equation}
It is clear that
$\Phi(x)=\alpha\|x\|_{\ell_1}+\frac{\lambda}{2}\|x\|_{\ell_2}^2$ is proper, convex, lower semi-continuous and coercive in $\ell_2$. Unfortunately, $ G(x)=\frac{1}{2}\| F(x)-y^{\delta}\|_Y^2-\left(\frac{\lambda}{2}\|x\|_{\ell_2}^2+\beta\|x\|_{\ell_2}\right)$ is not continuously Fr\'echet differentiable at $x=0$. So $G$ fails to fulfill the smoothness condition required by GCGM. Thus, the main difficulty in carrying out the minimization is how to impose the restriction on the smoothness of $G$. In this paper, we prove the convergence under a weaker condition, namely, $G'$ is continuous only on a closed set $S$, where $0\notin S$. We propose a numerical algorithm which is divided into two steps. In the first step, i.e.\ when $x^k=0$, the minimization problem is solved by the classical $\ell_1$ sparsity regularization. In the second step, i.e.\ when $x^k\neq 0$, the minimization of (\ref{equ3_2}) is solved by GCGM. We call it ST-$({\alpha \ell_1-\beta \ell_2})$ algorithm which is summarized in Algorithm 2.

\begin{algorithm}\label{alg2}
\caption{ST-$({\alpha \ell_1-\beta \ell_2})$ algorithm for problem (\ref{equ1_2}) in the finite dimensional space $\mathbb{R}^n$}
\begin{algorithmic}
\STATE{Choose $x^0\in \mathbb{R}^n$ such that $\Phi(x^0)<+\infty$, and set $k=0$.}
\STATE{for $k$ = 0, 1, 2, $\cdots$ do}
\STATE{~~~~If $x^k=0$ then}
\STATE{~~~~~~$\displaystyle x^{k+1}=\mathbb{S}_{\alpha/\lambda}\left(x^k-\frac{1}{\lambda}F'(x^k)^*(F(x^k)-y^{\delta})\right)$\quad (by the classical iterative soft thresholding algorithm)}
\STATE{~~~~else}
\STATE{~~~~~~Determine a solution $z^k$ by sloving \[\min\limits_{z \in \mathbb{R}^n} \langle G^{\prime}(x^k),z\rangle+\Phi(z).\]}
\STATE{~~~~~~Set a step size $s^k$ as a solution of \[\min\limits_{s \in [0,1]} G(x^k+s(z^k-x^k))+\Phi(x^k+s(z^k-x^k)).\]}
\STATE{~~~~~~$x^{k+1}=x^k+s^k(z^k-x^k)$}
\STATE{~~~~end if}
\STATE{~~~~$k=k+1$}
\STATE{end for}
\end{algorithmic}
\end{algorithm}


\subsection{Convergence analysis}

First, we recall two results proved in \cite{BLM09}.

\begin{lemma}\label{lemma3_1}
Let $G:\ell_2 \rightarrow \mathbb{R}$ denote a G\^{a}teaux-differentiable functional and let $\Phi: \ell_2\rightarrow \mathbb{R}$ be proper, convex, lower semi-continuous and coercive. Then, the first order necessary condition for optimality
in \eqref{equ3_2} is
\begin{equation}\label{equ3_3}
\displaystyle
x\in \ell_2:\quad \langle G'(x),y-x\rangle\geq \Phi(x)-\Phi(y) \quad for~all\quad y\in \ell_2.
 \end{equation}
This condition is equivalent to
\begin{equation}\label{equ3_4}
\displaystyle
\langle G'(x),x\rangle+\Phi(x)= \min\limits_{y\in\ell_2}(\langle G'(x),y\rangle+\Phi(y)).
 \end{equation}
\end{lemma}

\begin{lemma}\label{lemma3_2}
Let $\Phi$ be proper, convex, lower semi-continuous and coercive, and let $F$ be continuously Fr\'{e}chet differentiable.
Then
\begin{equation}\label{equ3_5}
 \begin{array}{llc}
\displaystyle
\Psi(x):=\langle G'(x),x\rangle+\Phi(x)-\min\limits_{y\in\ell_2}(\langle G'(x),y\rangle+\Phi(y))
 \end{array}
 \end{equation}
 is lower semi-continuous.
\end{lemma}

Next, we show that $\mathcal{J}_{\alpha,\beta}^{\delta}(x^k)$ decreases with respect to $k$, where $\{x^k\}$ is generated by Algorithm 2.

\begin{lemma}\label{lemma7}
Denote by $\{x^k\}$ the sequence generated by Algorithm 2. If $x^k$ fails to satisfy the first order optimality condition \eqref{equ3_3}, then $\mathcal{J}_{\alpha,\beta}^{\delta}(x^{k+1})\leq\mathcal{J}_{\alpha,\beta}^{\delta}(x^k).$
\end{lemma}
\begin{proof}
If $x^k=0$, by Algorithm 2, we have
\begin{align*}
\mathcal{J}_{\alpha,\beta}^{\delta}(x^{k+1})
&=G(x^{k+1})+\Phi(x^{k+1})=\frac{1}{2}\|F(x^{k+1})-y^{\delta}\|_{\ell_2}^2+\alpha\|x^{k+1}\|_{\ell_1}-\beta\|x^{k+1}\|_{\ell_2}\\
&\leq \frac{1}{2}\| F(0)-y^{\delta}\|_{\ell_2}^{2}+\alpha\| 0\|_{\ell_1}-\beta\| x^{k+1}\|_{\ell_2}\leq \frac{1}{2}\| F(0)-y^{\delta}\|_{\ell_2}^{2}+\alpha\| 0\|_{\ell_1}-\beta\| 0\|_{\ell_2}\\
& = \mathcal{J}_{\alpha,\beta}^{\delta}(x^k).\label{equ76}
\end{align*}
If $x^k\neq 0$, $G$ is Fr\'echet differentiable at $x^k$ and the rest of the proof is similar to that
of Lemma 2 in \cite{BLM09}.
\end{proof}

\begin{remark}\label{remark4}
Note that if $0=x^0=x^1$, we stop the iteration and 0 is inversion solution. Otherwise, by Lemma \ref{lemma7}, we see that
\[\mathcal{J}_{\alpha,\beta}^{\delta}(x^{1})-\mathcal{J}_{\alpha,\beta}^{\delta}(x^0)\leq -\beta\|x^1\|_{\ell_2}<0.\]
Since $\mathcal{J}_{\alpha,\beta}^{\delta}(x^k)$ decreases, $x^k\neq0$ for $k\ge1$. So we let $x^k\ne0$ whenever $k\ge1$.
\end{remark}

\begin{remark}\label{remark5}
By Lemma \ref{lemma7}, $\{\|x^k\|_{\ell_2}\}$ is bounded. Let us show that $\inf\|x^k\|_{\ell_2}>0$.
By Remark \ref{remark4}, $x^k \neq 0$. If $\inf\|x^k\|_{\ell_2}=0$, then
there exists a subsequence of $\{x^{k}\}$, still denoted by $\{x^{k}\}$, such that $\lim_{k\rightarrow \infty} x^k=0$.
Since $\mathcal{J}_{\alpha,\beta}^{\delta}(x)$ is weakly lower semi-continuous,
\[\liminf\limits_k\mathcal{J}_{\alpha,\beta}^{\delta}(x^k)\geq \mathcal{J}_{\alpha,\beta}^{\delta}(0).\]
Meanwhile, by Lemma \ref{lemma7}, we have
\[\liminf\limits_k\mathcal{J}_{\alpha,\beta}^{\delta}(x^k)\leq \mathcal{J}_{\alpha,\beta}^{\delta}(0).\]
So $\liminf\limits_k\mathcal{J}_{\alpha,\beta}^{\delta}(x^k)= \mathcal{J}_{\alpha,\beta}^{\delta}(0)$ for all $k\in \mathbb{N}$.
It means that $\mathcal{J}_{\alpha,\beta}^{\delta}(x^k)$ does not decrease, then 0 is the iterative solution.
So $c:=\inf\|x^k\|_{\ell_2}>0$. Thus, $\|x^k\|_{\ell_2}\geq c>0$ for any $k$.
\end{remark}

\begin{lemma}\label{lemma3_5}
Denote by $\{x^k\}$ the sequence generated by Algorithm 2 and $S:=\{x\in \ell_2 \mid 0<\inf\|x^k\|_{\ell_2}\leq\|x\|_{\ell_2}\leq C\}$, where $C$ is an upper bound of $\{\|x^k\|_{\ell_2}\}$. Then $G'$ is uniformly continuous on $S$.
\end{lemma}
\begin{proof}
By the definition of $G$, we see that $G$ is Fr\'echet differentiable and
\[ G'(x)h=\langle F'(x)^{*}(F(x)-y^{\delta}),h\rangle-\Theta'(x)h. \]
Then, for all $x, y\in S$,
\begin{align*}
\|G'(x)-G'(y)\|_{L(\mathbb{R}^n,\mathbb{R})}&\le \|F(x)\|_{\ell_2}\|F'(x)-F'(y)\|_{L(\mathbb{R}^n,\mathbb{R}^m)}+\|F'(x)\|_{L(\mathbb{R}^n,\mathbb{R}^m)}\|F(x)-F(y)\|_{\ell_2}\\
&\quad{}+\|\Theta'(x)-\Theta'(y)\|_{\ell_2}+\|y^{\delta}\|_{\ell_2}\|F'(x)-F'(y)\|_{L(\mathbb{R}^n,\mathbb{R}^m)}.
\end{align*}
The continuity of $G'$ follows from the continuous Fr\'echet differentiability of $F$ and $\Theta$ and
the boundedness of $F$ on $S$. Since $S$ is closed, $G'$ is uniform continuous on $S$.
\end{proof}


\begin{lemma}\label{lemma3_8}
Denote by $\{x^k\}$ the sequence generated by Algorithm 2.
Then $\lim\limits_k\Psi(x^k)=0$.
\end{lemma}
\begin{proof}
Let $S$ be as in Lemma \ref{lemma3_5}. By Lemma \ref{lemma3_5}, $G'$ is uniformly continuous on $S$; in particular,
$G'$ is bounded on $S$. Thus, $\{G'(x^k)\}$ is bounded.  Meanwhile, the direction $z^k$ is a solution of
\[\min\limits_{z \in \mathbb{R}^n} \langle G'(x^k),z\rangle+\Phi(z);\]
then $\{\|z^k\|_{\ell_2}\}$ is bounded. The minimizing property of the line-search for $s$ and the
intermediate value theorem imply that
\begin{align}
&(G+\Phi)(x^{k+1})-(G+\Phi)(x^{k})\nonumber\\
&\qquad \leq G(x^k+s(z^k-x^k))-G(x^k)+\Phi(x^k+s(z^k-x^k))-\Phi(x^k)\nonumber\\
&\qquad \leq G(x^k+s(z^k-x^k))-G(x^k)+s(\Phi(z^k)-\Phi(x^k))\nonumber\\
&\qquad \leq -s\Psi(x^k)+s\langle G'(x^k+ts(z^k-x^k))-G'(x^k), z^k-x^k\rangle, \label{equ3_6}
\end{align}
where $t\in[0,1]$ and $\Psi$ is defined in (\ref{equ3_5}).  Since $\{\|z^k\|_{\ell_2}\}$ and $\{\|x^k\|_{\ell_2}\}$
are bounded, there exist $c_1>0$, $c_2>0$ such that $\|z^k\|_{\ell_2}\leq c_1$ and $\|x^k\|_{\ell_2}\leq c_2$.
By (\ref{equ3_6}), we obtain
\[
\Psi(x^k)\le\frac{(G+\Phi)(x^k)-(G+\Phi)(x^{k+1})}{s}+(c_1+c_2)\parallel G'(x^k+ts(z^k-x^k))-G'(x^k)\parallel_{L(\mathbb{R}^n,R)}.
\]
Since $G'$ is uniformly continuous on $S$, for small enough $s$,
\[\|G'(x^k+ts(z^k-x^k))-G'(x^k)\|_{L(\mathbb{R}^n,R)}<\frac{\epsilon}{2(c_1+c_2)}.\]
Then,
by (\ref{equ3_6}), we have
\begin{equation}
0\leq\Psi(x^k)\leq \frac{(G+\Phi)(x^k)-(G+\Phi)(x^{k+1})}{s}+\frac{\epsilon}{2}.
\end{equation}
Since $(G+\Phi)(x^k)$ converges, there exists a natural number $K>0$ such that for every natural number $k>K$,
we have $(G+\Phi)(x^k)-(G+\Phi)(x^{k+1})<s\epsilon/2$. Thus, there exist a natural number $K$,
for every natural number $k>K$, $0\leq \Psi(x^k)<\epsilon$, which completes the proof of the lemma.
\end{proof}

\begin{theorem}\label{theorem3_4}
Denote by $\{x^k\}$ the sequence generated by Algorithm 2. Then $\{x^k\}$ has a subsequence converging to a stationary point of the functional $\mathcal{J}_{\alpha,\beta}^{\delta}(x)$.
\end{theorem}
\begin{proof}
Since $\{\|x^k\|_{\ell_2}\}$ is bounded, there exist a constant $x^*$ and a convergent subsequence of $\{x^k\}$ ,
still denoted by $\{x^k\}$, such that $x^k\rightarrow x^*$.
By Lemma \ref{lemma3_2}, $x^k\rightarrow x^*$ implies that $\liminf_{k\rightarrow\infty} \Phi(x^k)\geq \Phi(x^*)$. Then, by Lemma \ref{lemma3_8}, we have $\Phi(x^*)=0$, which completes the proof of the theorem.
%
\end{proof}



\subsection{Determining a solution $z^k$}\label{sec3.2}

It is shown that in Algorithm 2, when $x^k=0$, we can compute $x^{k+1}$ by the classical soft thresholding iteration. Next, we discuss how to utilize iterative soft thresholding algorithm to derive the iterative scheme when $x^k\neq0$.
A crucial issue is how to determine the direction $z^k$.

The Fr\'{e}chet derivative of $G(x)$
is given by
\[G'(x)=F'(x)^{*}(F(x)-y^{\delta})-\lambda x-\frac{\beta x}{\|x\|_{\ell_2}}.\]
The descent direction $z^k$ in Algorithm 2 is given by
\begin{equation}\label{equationadd1}
 \begin{array}{llc}
\displaystyle
\min\limits_{z} \langle F'(x^k)^{*}(F(x^k)-y^{\delta})-\lambda x^k-\frac{\beta x^k}{\|x^k\|_{\ell_2}},z\rangle+\frac{\lambda}{2}\|z\|_{\ell_2}^2+\alpha\|z\|_{\ell_1}.
 \end{array}
 \end{equation}
The minimizer of (\ref{equationadd1}) can be computed explicitly componentwise. The $i^{\rm th}$ component of $z$
satisfies
\begin{equation}\label{equationadd12}
z_i+\frac{\alpha}{\lambda}\mathrm{sign}(z_i)=\left(x^k+\frac{\beta x^k}{\lambda \|x^k\|_{\ell_2}}
-\lambda^{-1}F'(x^k)^{*}(F(x^k)-y^{\delta})\right)_{i}.
\end{equation}
The solution of (\ref{equationadd12}) can be expressed by the soft threshold (ST) function $\mathbb{S}_{\alpha/\lambda}$ and $S_{\alpha/\lambda}$, where $\mathbb{S}_{\alpha/\lambda}(x)$ is defined by
\begin{equation}\label{equationadd2}
\begin{array}{llc}
\displaystyle \mathbb{S}_{\alpha/\lambda}(x)=\sum\limits_{i}S_{\frac{\alpha}{\lambda}}(x_i)e_i
 \end{array}
 \end{equation}
and ${S}_{\alpha/\lambda}(t)$, $t\in \mathbb{R}$, is defined by
\begin{align}
\displaystyle S_{\alpha/\lambda}(t)= \left\{\begin{array}{ll}
\displaystyle t-\frac{\alpha}{\lambda}~~~~{\rm if}~~~t\geq\frac{\alpha}{\lambda}, \\[2mm]
\displaystyle 0~~~~~~~~~~{\rm if}~~~|t|<\frac{\alpha}{\lambda}, \\[2mm]
\displaystyle t+\frac{\alpha}{\lambda}~~~~{\rm if}~~~t \leq-\frac{\alpha}{\lambda}.
\end{array}
\right. 	\label{equationadd3}
\end{align}

\begin{lemma}\label{lemmaadd2} If $x^k\neq 0$, then the minimizer of problem \eqref{equationadd1} is given by
\begin{equation}\label{equadd1}
 \begin{array}{llc}
\displaystyle
z^k=\mathbb{S}_{\alpha/\lambda}\left(\left(\frac{\beta}{\lambda \|x^k\|_{\ell_2}}+1\right)x^k-\frac{1}{\lambda}F'(x^k)^{*}(F(x^k)-y^{\delta})\right).
 \end{array}
 \end{equation}
\end{lemma}
\begin{proof}
The proof is along the line of Lemma 2.3 in \cite{BBLM07}. The problem (\ref{equationadd1}) is equivalent to the problem
\begin{equation}
 \begin{array}{llc}
\displaystyle
\min\limits_{z}\sum\limits_{i}\frac{\lambda}{2}\bigg| z_{i}-\left(x^k+\frac{\beta x^k}{\lambda \|x^k\|_{\ell_2}}-\lambda^{-1}F'(x^k)^{*}(F(x^k)-y^{\delta})\right)_{i}\bigg|^2+\alpha|z_i|.
 \end{array}
 \end{equation}
From a result in \cite[Chapter 10]{RW1998}, for every proper,
convex $g:\mathbb{R}\rightarrow \mathbb{R}$ and every $\lambda>0$,
\[\displaystyle (I+\frac{1}{\lambda}\partial(\alpha \|\cdot\|_{\ell_1}))^{-1}(x)
=\displaystyle \arg\min\limits_{\omega}\bigg\{\frac{\lambda}{2}|\omega-x|^2+g(\omega)\bigg\}.\]
Then we can determine the minimizer $z^k$ by
\begin{equation}\label{equ39}
 \begin{array}{llc}
\displaystyle
z^k=\sum\limits_{i}\left[(I+\frac{1}{\lambda}\partial(\alpha \|\cdot\|_{\ell_1}))^{-1}\left(\left(x^k+\frac{\beta x^k}{\lambda \|x^k\|_{\ell_2}}-\lambda^{-1}F'(x^k)^{*}(F(x^k)-y^{\delta})\right)_{i}\right)\right]\cdot e_i.
 \end{array}
 \end{equation}
Using definition (\ref{equationadd2}) and (\ref{equationadd3}), we can rewrite (\ref{equ39})
in the form of (\ref{equadd1}).
\end{proof}

The complete ST-$({\alpha \ell_1-\beta \ell_2})$ algorithm is shown in Algorithm 3. Note that if $\beta=0$, (\ref{equadd1}) reduces to the standard iterative soft thresholding algorithm.

\begin{algorithm}\label{alg3}
\caption{ST-$({\alpha \ell_1-\beta \ell_2})$ algorithm for problem (\ref{equ1_2}) in the finite dimensional space $\mathbb{R}^n$}
\begin{algorithmic}
\STATE{Choose $x^0\in \mathbb{R}^n$ such that $\Phi(x^0)<+\infty$, and set $k=0$.}
\STATE{for $k$ = 0, 1, 2, $\cdots$ do}
\STATE{~~~~If $x^k=0$ then}
\STATE{~~~~~~$\displaystyle x^{k+1}=\arg\min\frac{1}{2}\|F(x)-y^{\delta}\|_{\ell_2}^{2}+\alpha\|x\|_{\ell_1}$}
\STATE{~~~~else}
\STATE{~~~~~~Determine a solution $z^k$ by sloving \[\displaystyle z^{k}=\mathbb{S}_{\alpha/\lambda}\left(\left(\frac{\beta}{\lambda \|x^k\|_{\ell_2}}+1\right)x^k-\frac{1}{\lambda}F'(x^k)^*(F(x^k)-y^{\delta})\right)\]}
\STATE{~~~~~~Set a step size $s^k$ as a solution of \[\min\limits_{s \in [0,1]} G(x^k+s(z^k-x^k))+\Phi(x^k+s(z^k-x^k)).\]}
\STATE{~~~~~~$x^{k+1}=x^k+s^k(z^k-x^k)$}
\STATE{~~~~end if}
\STATE{~~~~$k=k+1$}
\STATE{end for}
\end{algorithmic}
\end{algorithm}


\section{Numerical experiments}\label{sec5}

In this section, we implement the algorithm described in Section \ref{sec4} for a nonlinear compressive sensing (CS)
problem (\cite{BE13,B13,CL16,S12,YWLJWJ15}). Here we are interested in the sparse recovery for a CS problem where
the observed signal is measured with some nonlinear system. The research of nonlinear CS is not only important
in theoretical analyses but also in many applications, where the observation system is often nonlinear.
For example, in diffraction imaging, charge coupled device (CCD) records the amplitude of the Fourier transform of the original signal. So one only obtains the nonlinear measurements of the original signal. Fortunately, in \cite{B13}, it is shown that if the system satisfies some nonlinear conditions then recovery should still be possible.

Under the nonlinear CS frame, the measurement system is nonlinear. Assume, therefore, that the observation model is
\begin{equation}\label{ncsequ}
y=F(x)+\delta,
\end{equation}
where $\delta\in \mathbb{R}^m$ is a noise level, $x\in \mathbb{R}^n$ and $F: \mathbb{R}^n\rightarrow \mathbb{R}^m$ is a nonlinear operator. It is shown that if the linearization of $F(x)$ at an exact solution $x^{\dag}$ satisfies the restricted isometry property (RIP), then the convergence property is guaranteed (\cite{B13}). Next we illustrate the efficiency of the proposed algorithm by a nonlinear CS example of the form
\begin{equation}\label{ncsequ2}
y=F(x):=\hat{a}(A\hat{b}(x))
\end{equation}
which was introduced in \cite{B13}, where $A$ is CS matrix, $\hat{a}(\cdot)$ and $\hat{b}(\cdot)$ are
nonlinear operators, respectively. Here, $\hat{a}(\cdot)$ encodes nonlinearity after mixing by $A$ as well as
nonlinear ``crosstalk'' between mixed elements. $\hat{b}(\cdot)$ encodes the same system properties for the inputs
before mixing. For simplicity, we write $\hat{a}(x)= x+a(x)$ and $\hat{b}(x)= x+b(x)$, where again $a(x)$ and $b(x)$
are nonlinear maps. In particular, one always let $a(x)=x^c$ and $b(x)=x^d$, where $c, d \in \mathbb{N}_{+}$.
In \cite{S12}, the author stated that an important example in nonlinear compressive sensing is the case where
we observe signal intensities, i.e.\ the measurements are of the form
$y=(Ax)^2:=\left((Ax)_1^2,(Ax)_2^2, \cdots, (Ax)_i^2, \cdots\right)$, where $(Ax)_i$ is the $i^{\rm th}$
component of $Ax$. It is a particular case of the nonlinear observation system $y=\hat{a}(A\hat{b}(x))$,
where $\hat{a}(x)=x^2$ and $\hat{b}(x)=x$. In this case, the phase information
is missing. The problem is then to reconstruct the exact sparse signal $x^{\dag}$ from intensity measurements only.

Next, we turn to studying the nonlinearity of the operator $F(x)$. In \cite{B13}, it is shown that the Jacobian matrix of $\hat{a}(A\hat{b}(x))$ is of the form
\[ F'(x)=[I+a'_x][A(I+b'_x)],\]
where $a'_x$ is the Jacobian of $a(\cdot)$ evaluated at $A\hat{b}(x)$ and $b'_x$ is
the Jacobian of $b(\cdot)$ evaluated at $x$. We assume that $\|x\|_{\ell_2}$ and CS matrix $A$ are bounded, which implies that $a'_x$ is bounded. Then, there exists
a constant $c>0$ such that $\|F'(x_1)-F'(x_2)\|_{L(\mathbb{R}^n, \mathbb{R}^m)}\leq c\|x_1-x_2\|_{\ell_2}$, i.e.
$F'(x)$ is Lipschitz continuous, see the reference \cite[Lemma 3, Lemma 4]{B13}.

We present several numerical tests which demonstrate the efficiency of
the proposed method. To make Algorithm 3 clear to the reader, we study the influence of the parameters $\lambda$, $\eta$, $s_k$ and the nonlinear maps $a(\cdot)$ and $b(\cdot)$  on the inversion result $x^{*}$. Note that if $\eta=0$ i.e.\ $\beta=0$, (\ref{equ1_2}) reduces to the convex $\ell_1$ sparsity regularization.
Then (\ref{equadd1}) reduces to the form
\[ z^k=\mathbb{S}_{\alpha/\lambda}\left(x^k-\frac{1}{\lambda}F'(x^k)^*(F(x^k)-y^{\delta})\right).  \]
For the numerical simulation, we use a setting that $A$ is a Gaussian random measurement matrix. The nonlinear CS problem
is of the form $\hat{a}(A_{m\times n}\hat{b}(x_n))=y_m$, where $A_{m\times n}$ is a Gaussian random measurement matrix. The exact solution $x^{\dag}$ is $s$-sparse. The exact data $y^{\dag}$ is obtained by $y^{\dag}=\hat{a}(A\hat{b}(x^{\dag}))$.
White Gaussian noise is added to the exact data $y^{\dag}$ and $\delta$ is the noise level, measured in dB. The iterative solution is denoted by $x^{*}$. The performance of the iterative solution $x^*$ is evaluated by signal-to-noise ratio (SNR) which is defined by
\[  \mathrm{SNR}:=\displaystyle -10\log_{10}\frac{\|x^{*}-x^{\dag}\|_{\ell_2}^2}{\|x^{\dag}\|_{\ell_2}^2}.   \]

We utilize the discrepancy principle to choose the regularization parameter $\alpha$.
Given an initial regularization $\alpha$, if the regularization parameter $\alpha$ satisfies the discrepancy
principle $\|F(x^*)-y\|_{\ell_2}>\delta$, we try $\displaystyle \alpha_j=\frac{\alpha}{2^j}$, $j=1, 2, \cdots$. With $j$ increasing, we calculate $x^*$ until we find $\alpha=\inf\{\alpha>0 \mid \|F(x^*)-y^{\delta}\|_{\ell_2}>\delta\}$.

We let $n = 200$, $m = 0.4n$, $s = 0.2m$. For the sparsity regularization of linear ill-posed problems,  the value of $\|A_{m\times n}\|_2$ needs to be less than 1 (\cite{DDD04}). This requirement is still needed for the nonlinear CS problem \eqref{ncsequ}. Actually, the value of $\|A_{m\times n}\|_2$ is around 20, and Algorithm 3 is divergent without preprocessing. So to ensure the convergence, we need to re-scale the matrix $A_{m\times n}$ by
$A_{m\times n}\rightarrow0.05A_{m\times n}$. Note that Algorithm 3 will also be divergent with too small value of
$\|A_{m\times n}\|_2$.  We let the parameters $\lambda = 4.0$, $s^k=1$. The initial value
$x^0$ in Algorithm 3 is generated by calling $x^0$=1e-6$*\mathrm{ones}(n,1)$ in MATLAB. Actually, for sparse recovery, one natural choice for the initial value $x^0$ is $\mathbf{0}$ vector, i.e. $\mathrm{zeros}(n,1)$. However, (\ref{equadd1})
is not defined when $x^0=\mathbf{0}$. Another natural choice for the initial value $x^0$ is $\mathbf{1}$ vector, i.e.\
$\mathrm{ones}(n,1)$. However, the proposed algorithm is not convergent due to the absence of the multi-local minimums.

In the first test, we discuss the convergence and convergence rate of the proposed algorithm. We let $c=2$ and $d=3$.
The noise level $\mathrm{\delta}$ is $30\mathrm{dB}$.  We choose different parameters $\eta$ to test its influence
on the iterative solution $x^*$. Fig.\ \ref{fig1} shows the graphs of the iterative solution $x^{*}$ when the
regularization parameter $\alpha=0.125$. It is obvious that the results of inversion get better with $\eta$
increasing, which shows that the non-convex regularization with $\eta>0$ has better performance than the classical
$\ell_1$ regularization. Fig.\ \ref{fig2} displays graphs of the inversion solution $x^*$ with respect to
iteration number $k$ when $\eta=1.0$. It shows that the convergence is good. The algorithm 3 does not divergent
even with large enough iteration number $k$. Fig.\ \ref{fig3} shows the convergence rate of inversion solution
$x^*$ with respect to iteration number $k$. We use relative error to evaluate the performance of $x^*$,
where the relative error $e$ is defined by $e:=\|x^{*}-x^{\dag}\|_{\ell_2}/\|x^{\dag}\|_{\ell_2}$.

\begin{figure}[tbhp]
\centering
\subfigure[Exact signal]{\includegraphics[width=70mm,height=30mm]{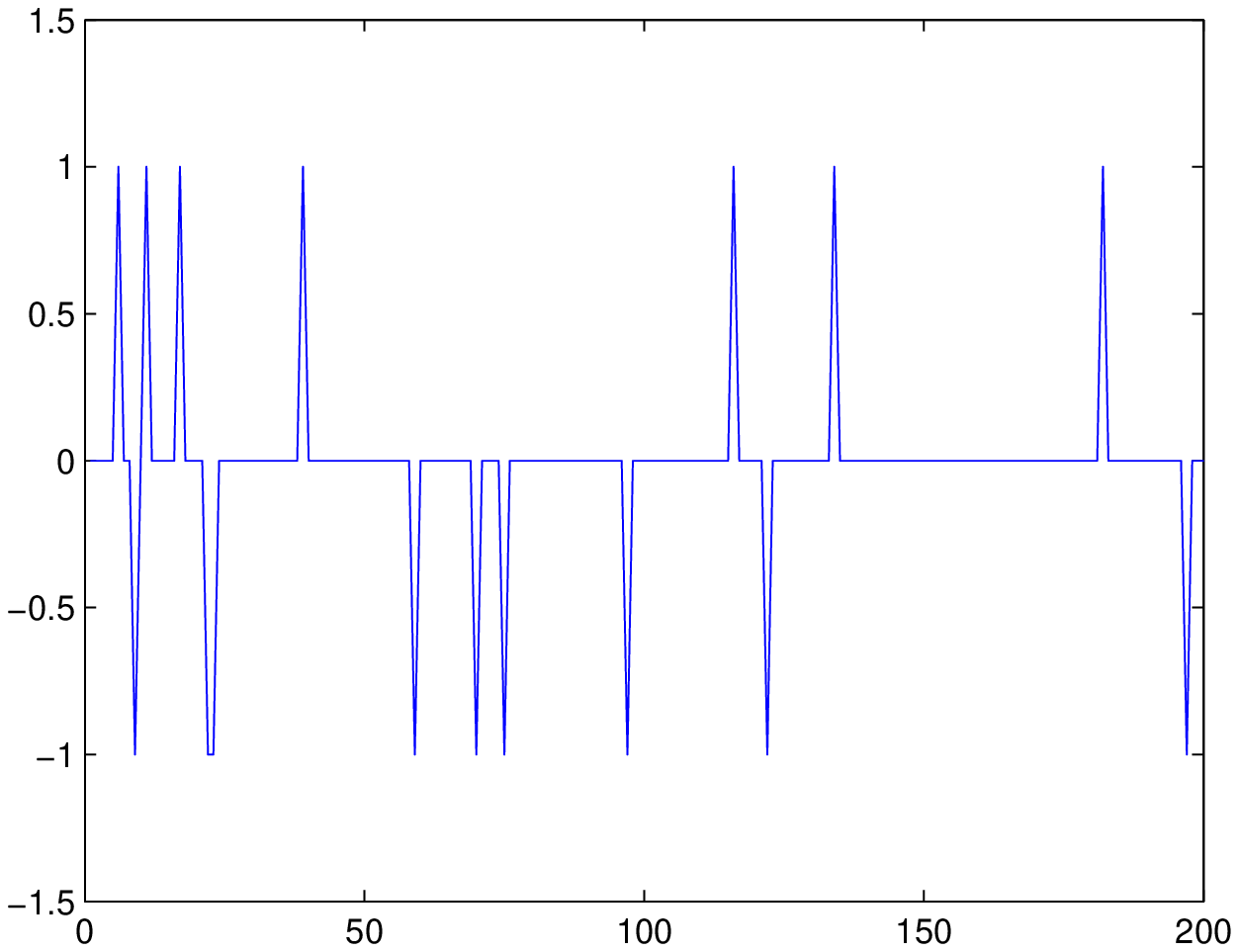}}
\subfigure[Observed and exact data ($\delta$=30dB)]{\includegraphics[width=70mm,height=30mm]{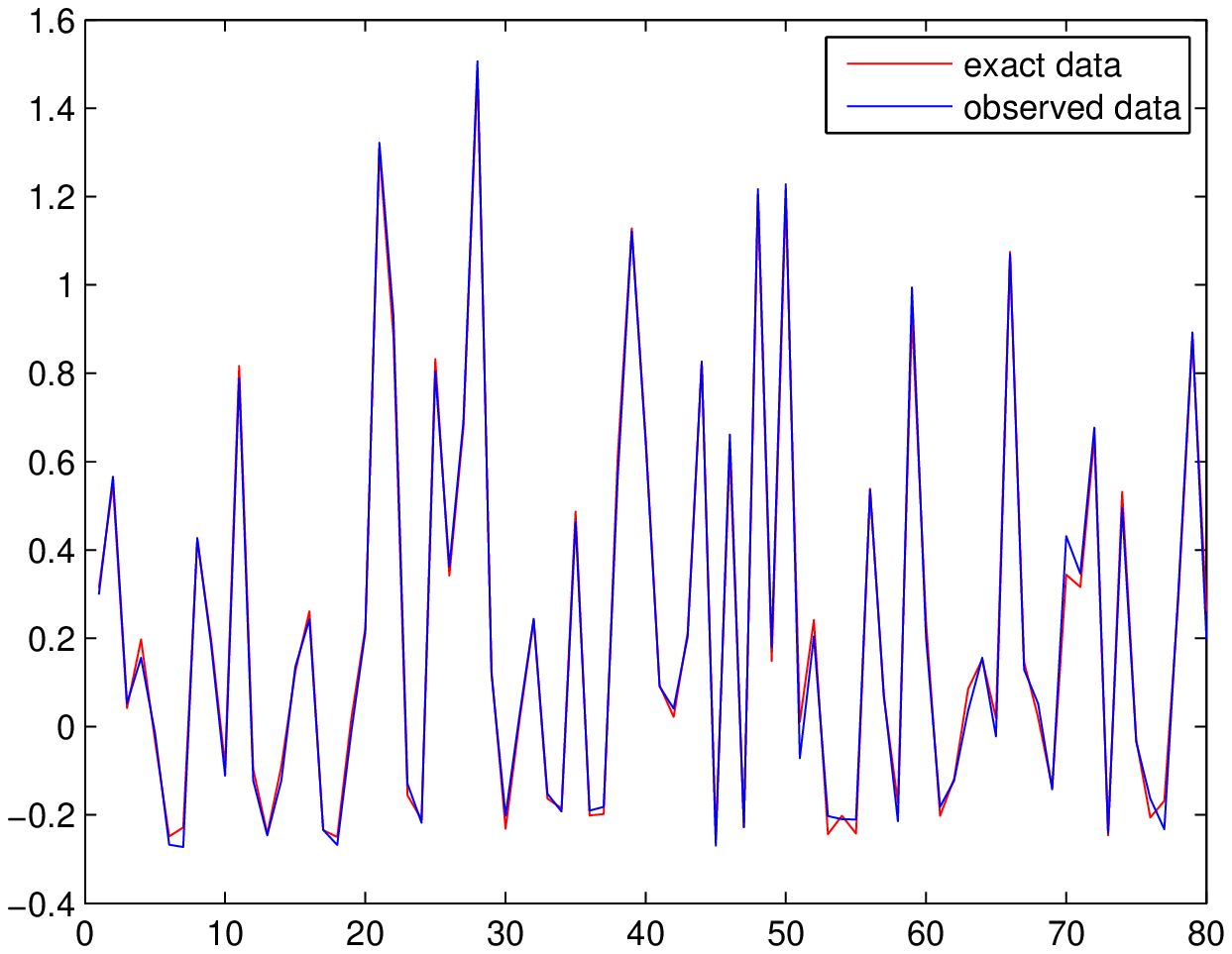}}\\
\subfigure[$\eta=0.0$, SNR=29.4755]{\includegraphics[width=70mm,height=30mm]{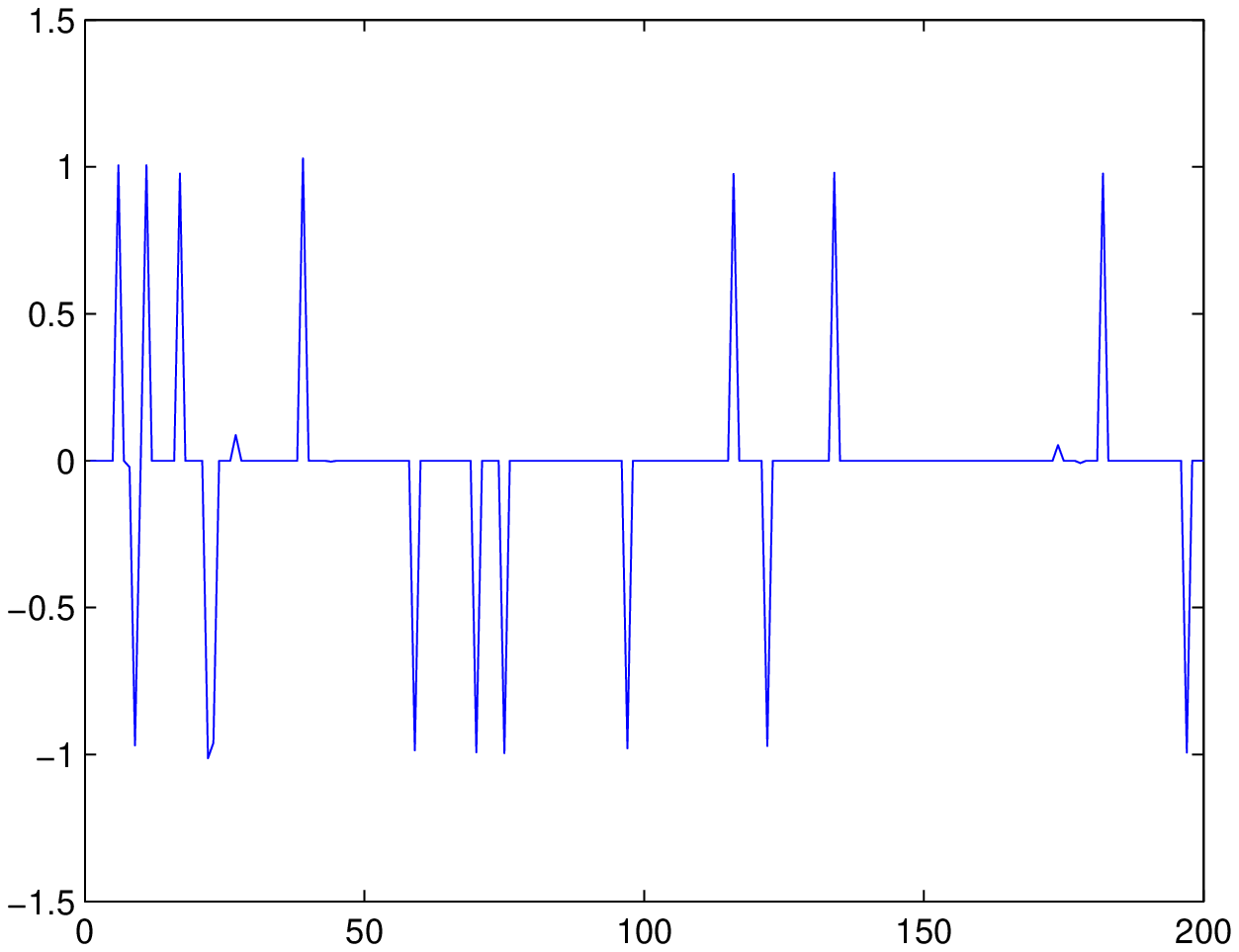}}
\subfigure[$\eta=0.4$, SNR=34.3366]{\includegraphics[width=70mm,height=30mm]{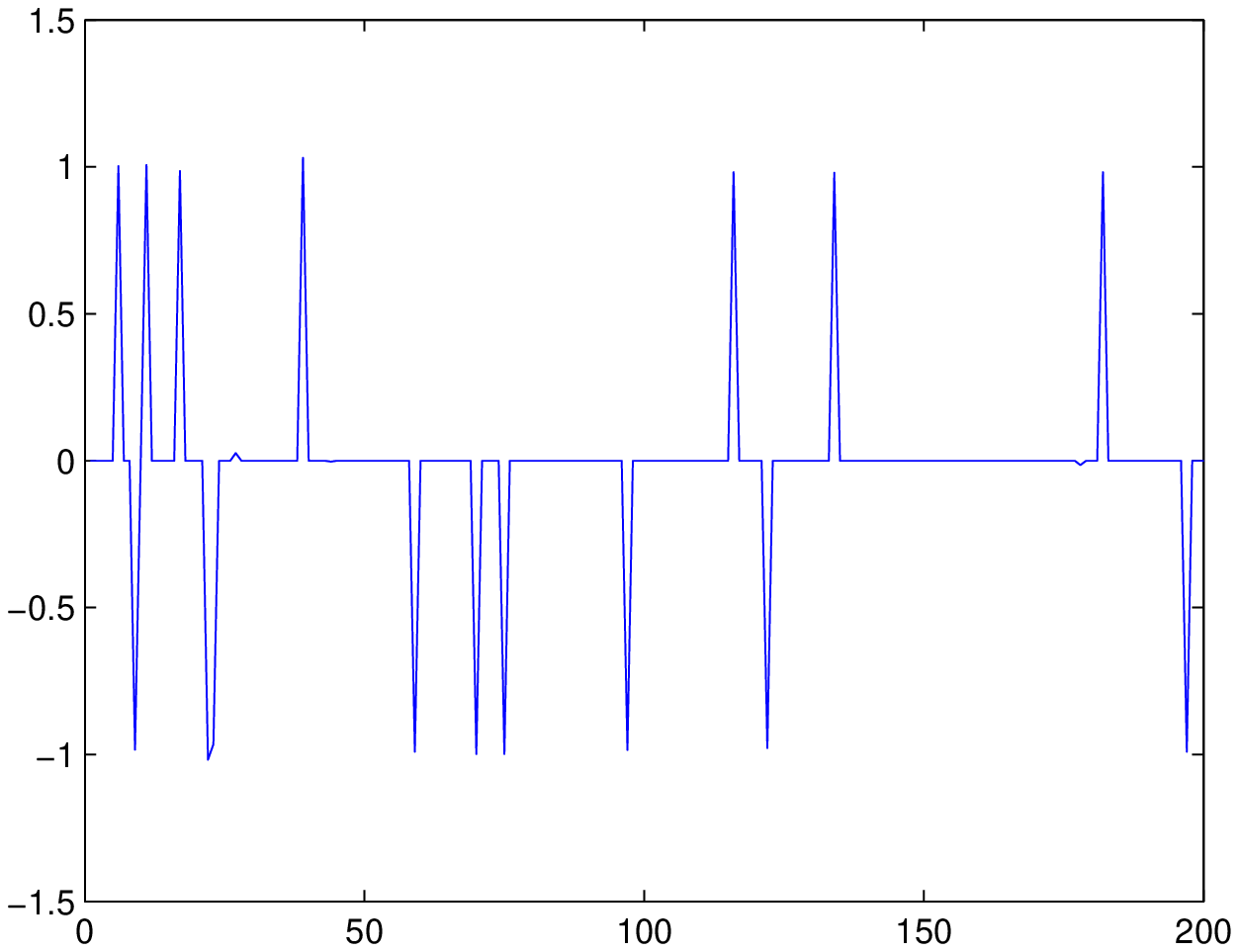}}\\
\subfigure[$\eta=0.8$, SNR=35.0716]{\includegraphics[width=70mm,height=30mm]{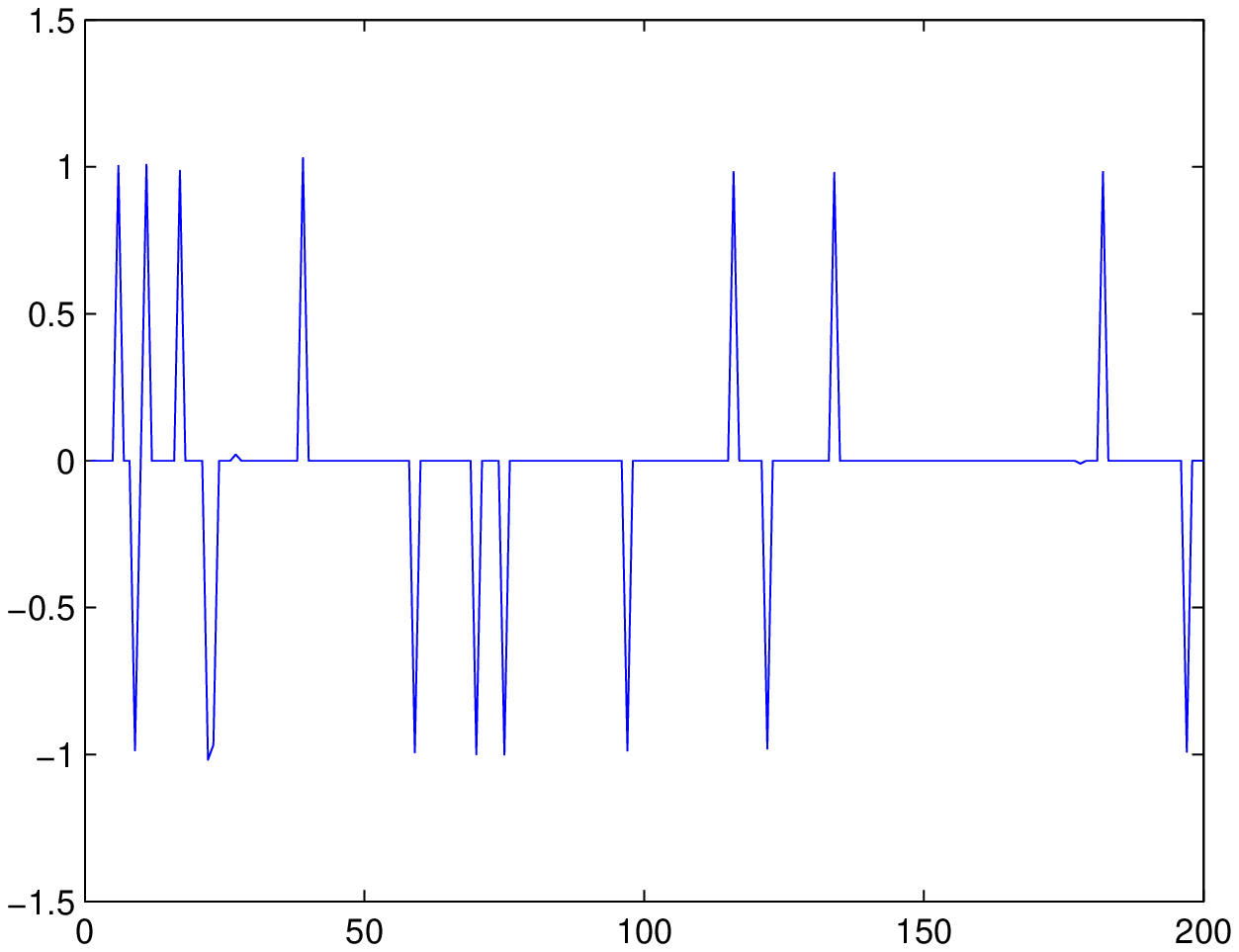}}
\subfigure[$\eta=1.0$, SNR=35.3914]{\includegraphics[width=70mm,height=30mm]{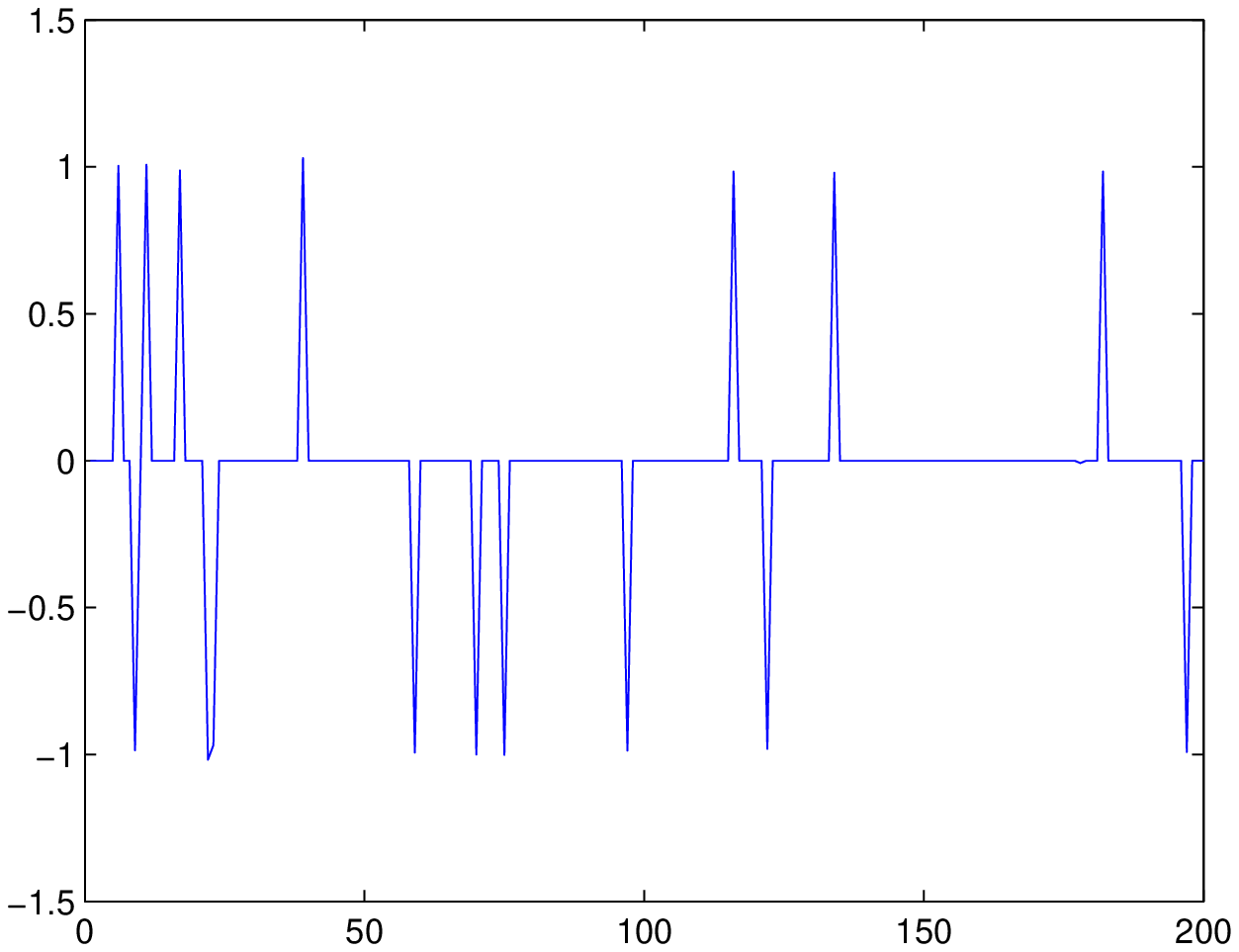}}\\
\caption{(a) Exact signal. (b) Observed and exact data. (c)--(f) The inversion solution $x^*$ with different $\eta$ at a fixed regularization parameter $\alpha=$ 0.125.}
\label{fig1}
\end{figure}

\begin{figure}[tbhp]
\centering
\subfigure[$k=5$, SNR=1.8169]{\includegraphics[width=70mm,height=30mm]{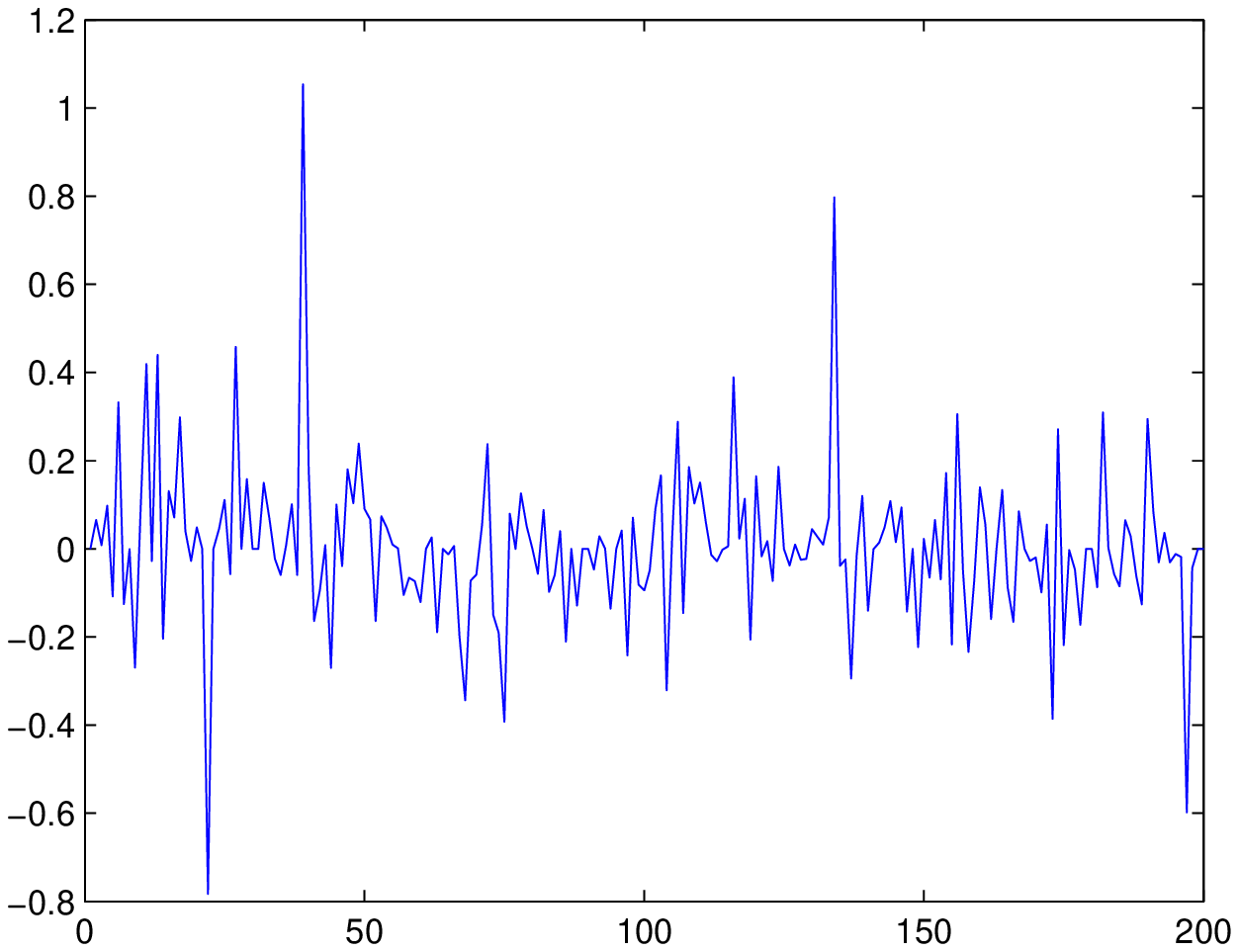}}
\subfigure[$k=40$, SNR=4.2687]{\includegraphics[width=70mm,height=30mm]{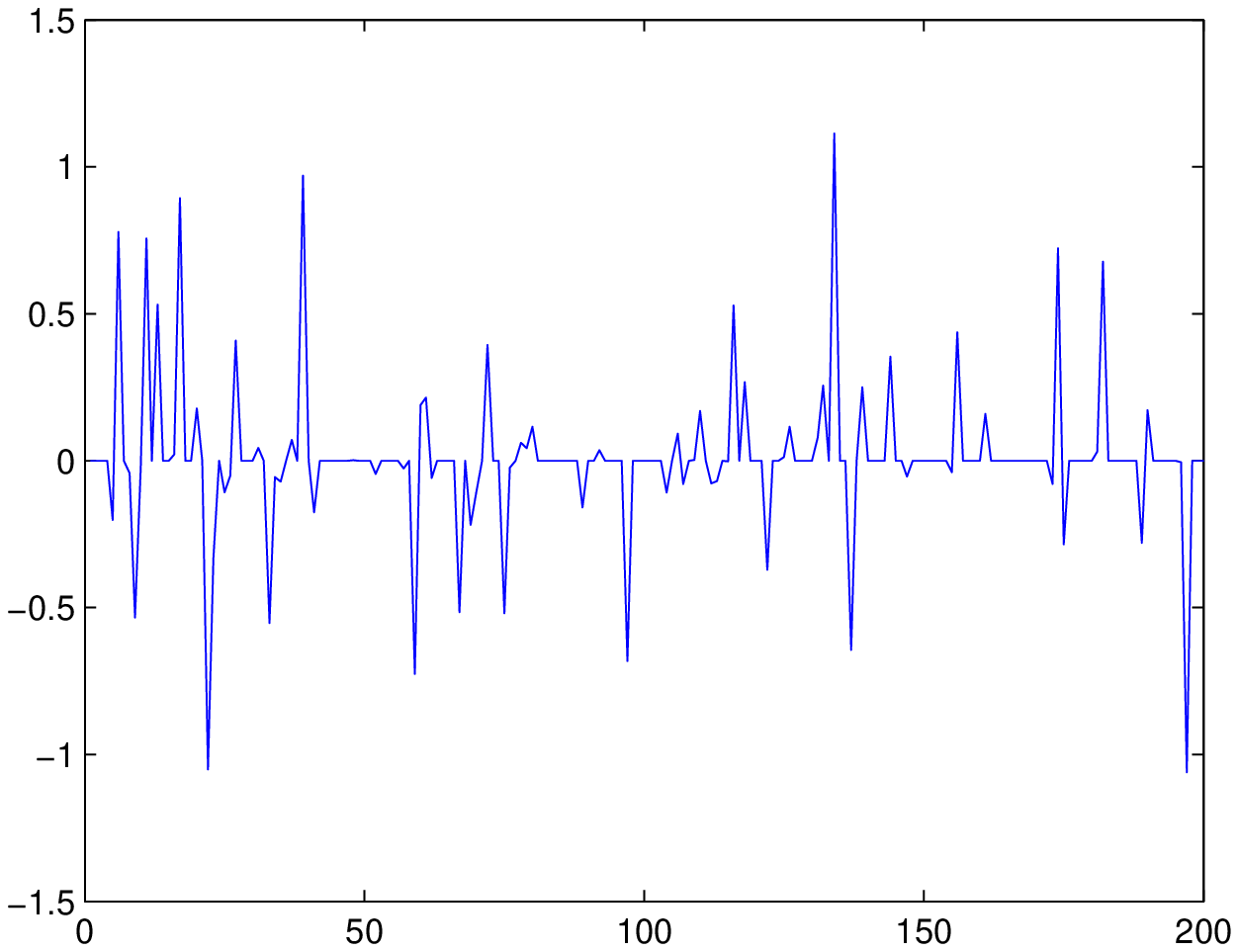}}\\
\subfigure[$k=80$, SNR=23.8996]{\includegraphics[width=70mm,height=30mm]{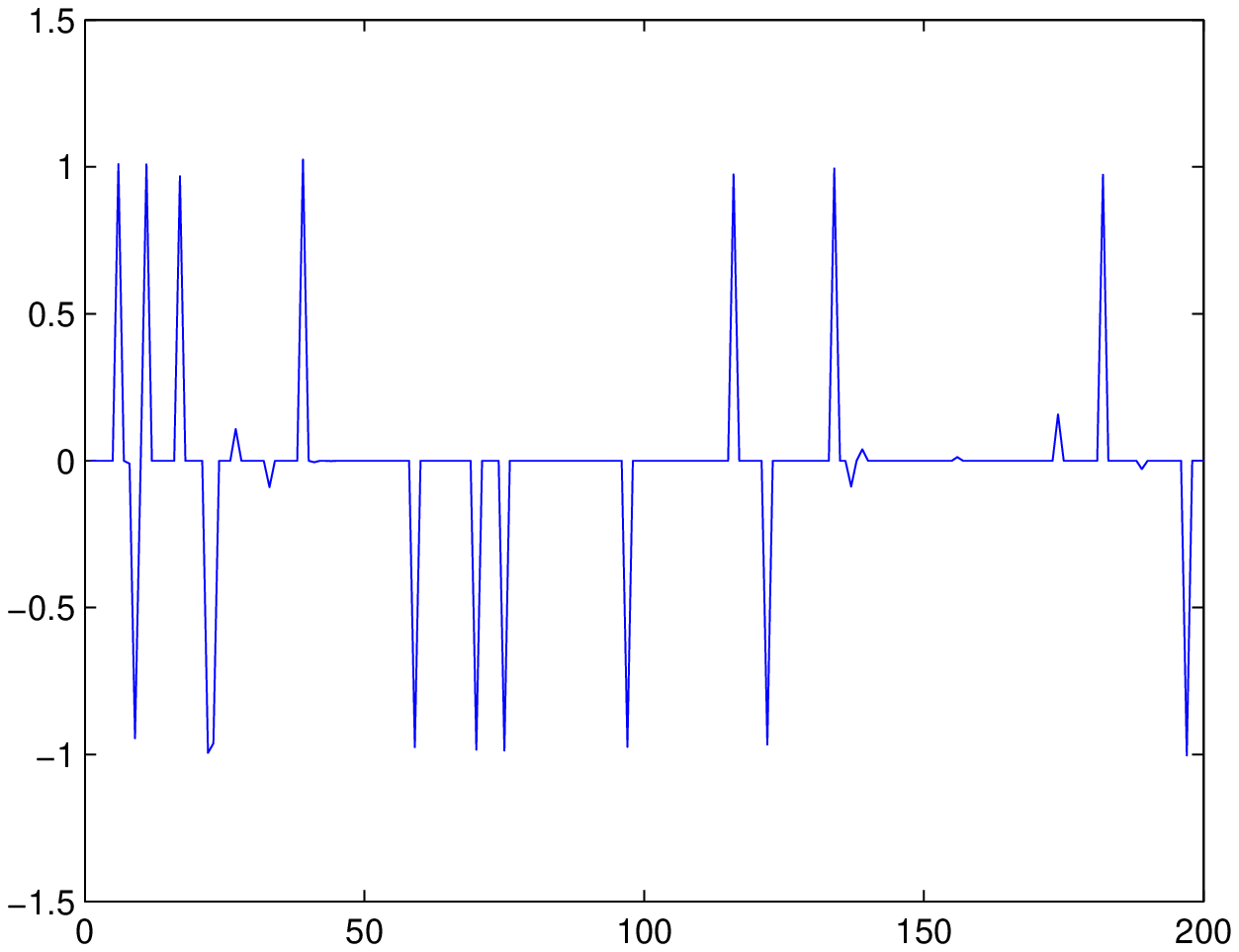}}
\subfigure[$k=100$, SNR=35.3914]{\includegraphics[width=70mm,height=30mm]{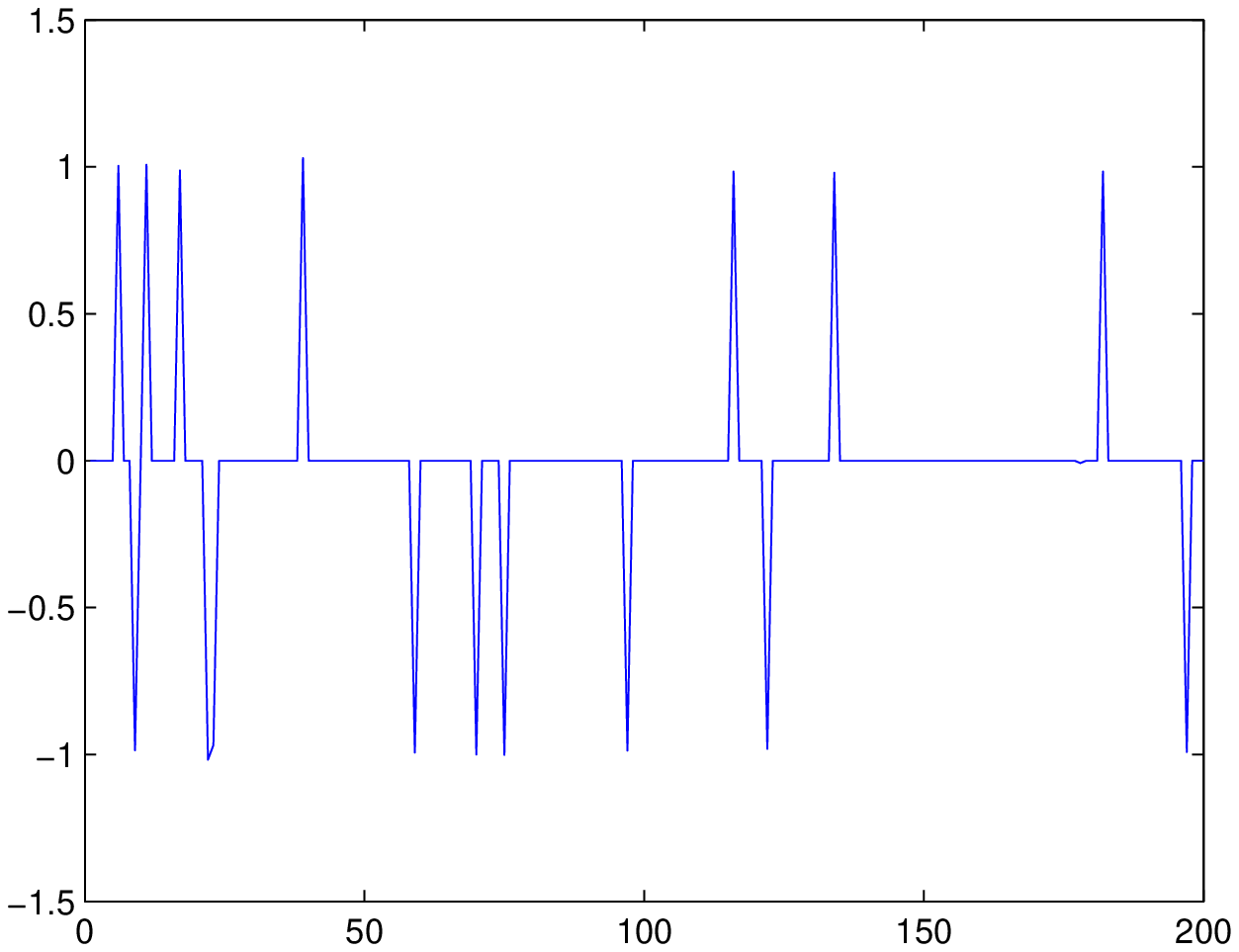}}\\
\caption{The inversion solution $x^*$ with different iteration number $k$ at a fixed parameter $\eta=$1.0.}
\label{fig2}
\end{figure}

Next, we test the effect of the step size $s^k$. For nonlinear ill-posed problems, it is shown in \cite{BBLM07} that
GCGM is convergent with a fixed step size. So, for the sake of simplicity in computation, we let $s^k=1$ in
this section. In Table \ref{tab1}, we set $\eta=1.0$ and let $\alpha$, $\delta$, $c$ and $d$ same as that in test 1.
We check the convergence and convergence rate of Algorithm 3 with different fixed step sizes.
Table \ref{tab1} shows that Algorithm 3 converges when the step size $s^k\leq 1.5$ and it is invalid when the
step size $s^k\geq 2.0$. Though Algorithm 3 is convergent when $s^k\leq 1.5$, it needs more iteration numbers
with smaller step size $s^k$, which implies that the convergence rate is worse with the step size $s^k$
getting smaller. Fortunately, Algorithm 3 provides same recovery results, when the step size $s^{k}\leq 1.5$.
So to ensure the convergence, one should choose a smaller step size $s^k$ at first. If Algorithm 3 converges,
one should choose the larger $s^k$ step by step to get the better convergence rate.
\begin{figure}[tbhp]
\centering
\includegraphics[width=140mm,height=40mm]{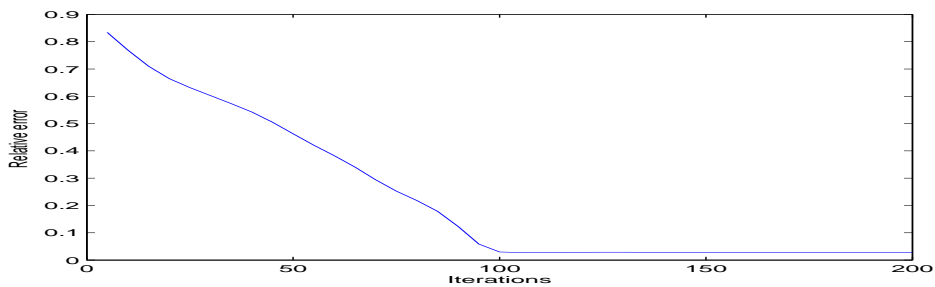}
\caption{Convergence rate of the inversion solution $x^*$ with respect to iteration number $k$ at a fixed parameter $\eta=$1.0.}
\label{fig3}
\end{figure}

\begin{table}[tbhp]
\caption{SNR of inversion solution $x^{*}$ and iterations of Algorithm 3 with different fixed step sizes}
\begin{tabular}{cccccccccccc} \hline
  & $s^k=0.001$&$s^k=0.01$& $s^k=0.1$& $s^k=1.0$& $s^k=1.5$&  $s^k=2.0$&  $s^k=3.0$\\
\hline
SNR &35.3914&35.3914&35.3914&35.3914&35.3914&NaN&NaN\\
iterations &1000&200&130&125&125&------&------\\
\hline
\end{tabular}
\label{tab1}
\end{table}

In the third test, we test the effect of the parameter $\lambda$. By \eqref{equ4_1}, it is obvious that the
inversion results do not change with respect to $\lambda$. However, numerical results show that larger or smaller parameter $\lambda$ leads to divergence (\cite{DH19}), which is still true
even for nonlinear ill-posed problems. Actually, from the formulation of Algorithm 3, we see that a small value
of $\lambda$ admits a larger $\beta x^k/(\lambda \|x^k\|_{\ell_2})$ in \eqref{equadd1}. Meanwhile, a larger value of
$\lambda$ admits a smaller value of the threshold $\alpha/\lambda$. In Table \ref{tab2}, we set
$\eta=1.0$ and let $\alpha$, $\delta$, $c$ and $d$ same as that in test 1 and give the inversion results
of Algorithm 3 with different $\lambda$. It is shown that $4.0\leq \lambda\leq5.0$ is a a good choice.
When $\lambda\leq 2.5$, Algorithm 3 is invalid. However, we can not get satisfactory inversion results
when $\lambda> 5$. Actually, SNR of the inversion solution $x^*$ decrease with $\lambda> 5$ increasing.

\begin{table}[tbhp]
\caption{SNR of inversion solution $x^{*}$ with different $\lambda$}
\begin{tabular}{cccccccccccc} \hline
$\displaystyle \lambda$&2& 2.5& 3.0& 3.5& 4.0& 4.5& 5.0& 6.0&7.5\\
\hline
SNR&NaN&NaN&26.9716&32.0468&35.3901&35.3914&35.3920&26.4582&17.2486\\
\hline
$\displaystyle \lambda$&10.0& 12.5& 15.0& 20.0& 25.0& 30.0& 35.0& 40.0& 45.0\\
\hline
SNR&12.2586&10.4873&9.2452&5.1721&3.7938&3.4259&3.5837&3.1547&2.4683\\
\hline
\end{tabular}
\label{tab2}
\end{table}

In the fourth test, we study the stability of Algorithm 3. To test the influence of $\delta$, we choose
several different noise levels which are added to the exact data $y^{\dag}$. Table \ref{tab3} displays
the inversion results. Obviously, the SNR of Inversion solution $x^{*}$ increase with the noise level decreasing.
It is shown that we can obtain satisfactory result when the noise level $\delta\geq 20\mathrm{dB}$.
However, Algorithm 3 does not converge for low noise levels, i.e.\ $\delta\leq 10\mathrm{dB}$. Meanwhile,
Table \ref{tab3} shows that Algorithm 3 has good stability corresponding to the noise level whenever the
parameter $\eta$ is. Which implies that the stability of Algorithm 3 is not sensitive with respect to $\eta$.

\begin{table}[tbhp]
\caption{SNR of inversion solution $x^{*}$ with several noise levels}
\begin{tabular}{cccccccccccc} \hline
& $ \eta=0$& $ \eta=0.2$& $ \eta=0.4$& $ \eta=0.6$& $ \eta=0.8$& $ \eta=1.0$\\
\hline
$\mathrm{Noise}~~\mathrm{free},\alpha=0.015$&44.5479&45.0044&45.4848&45.9917&46.5286&47.0994\\
$\delta$=50dB, $\alpha=0.031$&43.3370&45.7636&46.2097&46.6770&47.1678&47.6844\\
$\delta$=40dB, $\alpha=0.062$&36.2226&37.6898&38.1772&38.6865&39.2201&39.7804\\
$\delta$=30dB, $\alpha=0.125$&29.4775&32.1682&34.3366&34.9682&35.0716&35.3914\\
$\delta$=20dB, $\alpha=0.125$&22.4699&24.1333&25.5513&25.7172&25.7959&25.9081\\
$\delta$=10dB, $\alpha=0.250$&-1.5015&-1.5146&-1.5307&-1.5438&-1.5546&-1.5641\\

\hline
\end{tabular}
\label{tab3}
\end{table}

In the last test, we discuss the influence of the nonlinearity of $F$, i.e.\ the parameter $c$ and $d$
on the inversion solution $x^*$. The nonlinearity of the CS problem \eqref{ncsequ2} depends on the parameters
$c$ and $d$. In particular, the degree of nonlinearity of $F$ increases with the parameter $c$ and $d$ increasing.
In Table \ref{tab4}, we set $\eta=1.0$ and let parameters $\alpha$, $\delta$ and $s^k$ same as that in test 1.
It is obvious that the inversion results are stable with respect to the parameter $c$. The SNR of the
inversion solution $x^*$ are similar with different parameter $c$. However, the inversion results are sensitive
with respect to the parameter $d$. When $d\geq 7$, we can not get satisfactory results. In particular,
Algorithm 3 is invalid when the parameter $d$ is even number. Fig.\ \ref{fig4} shows the inversion solution
$x^*$ with respect to the iterations at the fixed parameter $c=2$ and $d=4$. Actually, Algorithm 3 can only
identify the positive impulses and it fails to recovery the negative impulses when $d$ is even number.

\begin{table}[tbhp]
\caption{ SNR of inversion solution $x^{*}$ with different parameters $c$ and $d$}
\begin{tabular}{cccccccccccc} \hline
      & $d=1$& $d=2$& $d=3$& $d=4$& $d=5$& $d=6$& $d=7$& $d=8$& $d=9$\\
\hline
$c=1$&46.2683&4.1589&41.8642&3.2158&41.2564&2.5784&13.5876&NaN&NaN\\
$c=2$&40.6830&4.2591&42.0519&3.0102&43.8885&1.2486&11.9139&NaN&NaN\\
$c=3$&39.0531&3.0102&40.5454&1.2494&33.1849&2.0412&14.1032&NaN&NaN\\
$c=4$&39.4280&2.0409&39.5473&2.0311&33.4812&3.0022&18.8408&6.0172&NaN\\
$c=5$&39.5284&3.0097&40.5428&2.0412&31.6905&4.2200&17.0556&2.0336&NaN\\
$c=6$&39.8423&3.0095&39.9291&2.0411&35.7897&1.2494&16.8327&2.0403&NaN\\
$c=15$&39.8423&1.2493&42.2085&3.0101&38.4862&2.0412&19.9787&3.0045&NaN\\
$c=20$&39.8423&1.2492&38.8362&2.0412&39.2417&3.0103&18.6429&2.0412&NaN\\
$c=50$&40.5934&2.1863&39.7846&2.1957&39.7341&2.9472&19.1584&2.6893&NaN\\
\hline
\end{tabular}
\label{tab4}
\end{table}

\begin{figure}[tbhp]
\centering
\subfigure[Exact signal]{\includegraphics[width=70mm,height=30mm]{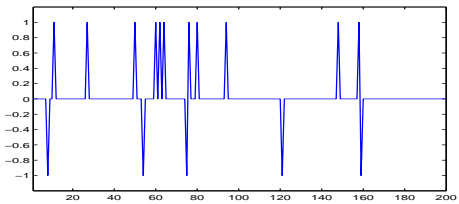}}
\subfigure[$k=40$, SNR=1.7690]{\includegraphics[width=70mm,height=30mm]{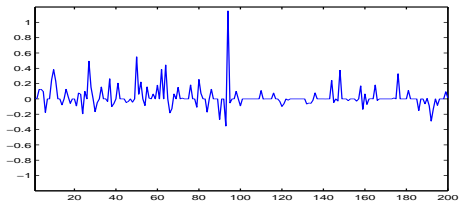}}\\
\subfigure[$k=80$, SNR=4.5921]{\includegraphics[width=70mm,height=30mm]{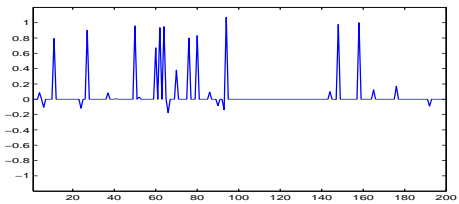}}
\subfigure[$k=100$, SNR=5.0444]{\includegraphics[width=70mm,height=30mm]{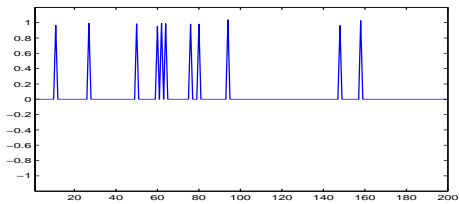}}\\
\caption{(a) Exact signal. (b)--(d) The inversion solution $x^*$ with different iteration number $k$ at a fixed constant $c=2$ and $d=4$.}
\label{fig4}
\end{figure}



\section{Conclusion}\label{sec6}

We analyzed the $\alpha\ell_1-\beta\ell_2$ $(\alpha\geq\beta\geq 0)$ sparsity regularization for nonlinear
ill-posed problems. For the well-posedness of the regularization, compared to the case $\alpha>\beta\geq 0$, we only
obtained the weak convergence for the case $\alpha=\beta\geq 0$. If the nonlinear operator $F'$ is Lipschitz continuous,
we proved that the regularized solution is sparse. Two different convergence rates
$O(\delta^{\frac{1}{2}})$ and $O(\delta)$ were obtained under two widely adopted nonlinear conditions. A soft
thresholding algorithm ST-($\alpha\ell_1-\beta\ell_2$) can be extended to solving the non-convex
$\alpha\ell_1-\beta\ell_2$ $(\alpha\geq\beta\geq 0)$ sparsity regularization for nonlinear ill-posed problems.
Numerical experiments show that the proposed method is convergent and stable. However, for some particular nonlinear CS problems, i.e. the parameter $d$ is an even number, we can only identify the positive impulses.

\end{document}